\documentclass[a4paper,12pt]{amsart}
\usepackage{amsfonts}
\usepackage{amssymb}
\usepackage[utf8]{inputenc}
\usepackage{amsmath}
\usepackage{pdflscape}
\usepackage{float}
\usepackage{easyReview}
\usepackage{graphicx}
\usepackage[colorlinks=true, linkcolor=red, citecolor=blue]{hyperref}
\usepackage[]{epsfig}
\usepackage[]{pstricks}
\usepackage{tikz}

\newtheorem{theorem}{Theorem}[section]
\newtheorem{proposition}[theorem]{Proposition}
\newtheorem{lemma}[theorem]{Lemma}

\newtheorem{remark}{Remark}
\theoremstyle{remark}

\newtheorem{definition}{Definition}
\usepackage{mathrsfs}
\def\supp{\operatorname{{supp}}}
\def\re{\mathrm{Re}}
\def\im{\mathrm{Im}}

\numberwithin{equation}{section}
\makeatletter
\@namedef{subjclassname@2010}{\textup{2020} Mathematics Subject Classification}
\makeatother

\begin{document}
	
	\title[Critical NLS: Local control results]{Control of the Schr\"{o}dinger equation in $\mathbb{R}^3$: The critical case}
	\author[Braz e Silva]{P. Braz e Silva$^\star$}
	\address{Departamento de Matem\'atica, Universidade Federal de Pernambuco, S/N Cidade Universit\'aria, 50740-545, Recife (PE), Brazil}
	\email{pablo.braz@ufpe.br}
	\author[Capistrano--Filho]{R. de A. Capistrano--Filho}
	\address{Departamento de Matem\'atica, Universidade Federal de Pernambuco, S/N Cidade Universit\'aria, 50740-545, Recife (PE), Brazil}
	\email{roberto.capistranofilho@ufpe.br}
	\author[Carvalho]{J. D. do N. Carvalho}
	\address{Departamento de Matem\'atica, Universidade Federal de Pernambuco, S/N Cidade Universit\'aria, 50740-545, Recife (PE), Brazil}
	\email{jackellyny.dassy@ufpe.br}
	\author[Dos Santos Ferreira]{D. dos Santos Ferreira}
	\address{Institut \'Elie Cartan de Lorraine, UMR CNRS 7502, \'equipe SPHINX, INRIA, Universit\'e de Lorraine, F-54506 Vandoeuvre-lès-Nancy Cedex, France.}
	\email{ddsf@math.cnrs.fr}
	\thanks{$^\star$Corresponding author: pablo.braz@ufpe.br}
	\thanks{Dedicated to Prof. Manuel González-Burgos and Prof. Francisco Guillén-González on the occasion of their 60th birthdays.}
	
	\begin{abstract}
		This article deals with the $H^{1}$--level local null controllability for the energy-critical nonlinear Schr\"{o}dinger equation in $\mathbb{R}^3$. Firstly, we demonstrate that the problem under consideration is well-posed using Strichartz estimates. Moreover, through the Hilbert uniqueness method, we prove the linear Schr\"{o}dinger equation to be controllable by means of a control force supported in the exterior of a ball. Finally, we use a perturbation argument and show local controllability for the critical nonlinear Schr\"{o}dinger equation. 
	\end{abstract}
	
	\date{\today}
	\maketitle

	\thispagestyle{empty}
	
	%***********************************************
	\normalsize
	\section{Introduction}
	The study of the energy-critical case of the nonlinear Schrödinger equation (C-NLS) in three dimensions, namely, 
	\begin{equation}\label{eq1aa}
		\left\{
		\begin{array}{lr}
			i\partial_{t} u + \Delta u  \pm |u|^{4}u =0 , \ \mbox{on }[0,+\infty)\times\mathbb{R}^{3},\\
			u(x,0) = u_{0} (x) \in H^{1}(\mathbb{R}^{3}),
		\end{array}
		\right.
	\end{equation} 
	is motivated by several deep mathematical and physical reasons. In the context of partial differential equations, the term ``energy-critical" refers to the fact that the associated energy of the equation is invariant with respect to its natural scaling. More precisely, if $u$ solves \eqref{eq1aa}, then so does $$u_{\lambda}(x,t)=\lambda^{1/2}u(\lambda x,\lambda^{2}t),$$ for every $\lambda>0$, and this scaling leaves invariant both the energy $$E(u)=\frac12\int_{\mathbb{R}^{3}}|\nabla u|^{2}\,dx\mp\frac16\int_{\mathbb{R}^{3}}|u|^{6}\,dx$$ and the $\dot H^{1}(\mathbb{R}^{3})$ norm. In three dimensions, the equation \eqref{eq1aa} is energy-critical, meaning the nonlinearity is precisely balanced with the dispersive nature of the equation. This balance leads to delicate analytical challenges, such as the need for sophisticated tools to study global well-posedness, scattering, and the behavior of solutions.
	
	Understanding whether solutions scatter (i.e., behave like linear solutions at infinity) is a central question for large data. The system \eqref{eq1aa} in three dimensions is a key case for studying the global dynamics of solutions, including whether all solutions with subcritical energy scatter or if some can lead to blow-up. %The four-dimensional case is often studied using concentration compactness methods, which are crucial for understanding the dynamics near critical points. These methods have broader implications for other dispersive equations as well.
	In summary, the energy-critical system \eqref{eq1aa} in three dimensions is a rich and challenging problem with deep connections to pure and applied mathematics, making it a significant and highly motivated area of study. For details, we strongly encourage the reader to see \cite{cazenave_book,KeMe,Visan,MeRa,Tao,KeMe1,cazenave1} and the references therein. In particular, for the defocusing equation in $\mathbb{R}^{3}$, Colliander, Keel, Staffilani, Takaoka and Tao \cite{CKSTT} proved global well-posedness and scattering for arbitrary data in $\dot H^{1}(\mathbb{R}^{3})$, while Kenig and Merle \cite{KeMe} established global well-posedness
and scattering for radial focusing data satisfying
$$E(u_0)<E(W)$$ and
$$\|\nabla u_0\|_{L^2}<\|\nabla W\|_{L^2}.$$ Our analysis only requires the small-data theory (see Theorem \ref{t2.4} below); however, the large-data theory of \cite{CKSTT} is the natural starting point for the global control problem in the defocusing case (see Section \ref{Sec4}).
	
	\subsection{Addressed issues and review of the literature}
	In this article, we will consider the ${H}^{1}(\mathbb{R}^{3})$ local controllability for the %defocusing% 
	energy-critical nonlinear Schr\"{o}dinger equation (C-NLS)
	\begin{equation}\label{eq1a}
		\left\{
		\begin{array}{lr}
			i\partial_{t} u + \Delta u  \pm |u|^{4}u =f(x,t) , \ \mbox{on }[0,+\infty)\times\mathbb{R}^{3},\\
			u(x,0) = u_{0} (x) \in {H}^{1}(\mathbb{R}^{3}),
		\end{array}
		\right.
	\end{equation}
	where $u = u(x, t)$ is a complex-valued function of two variables $x \in\mathbb{R}^3$ and $t\in\mathbb{R}$, while the function $f(x, t)$ is a control input.	We are mainly concerned with the following null control problem for system \eqref{eq1a}. 
	
	\vspace{0.2cm}
	
	\noindent\textbf{Control problem:} Let $T>0$ be given. For any given $u_0 \in {H}^{1}(\mathbb{R}^{3})$, can one find a control $f(x,t)$ such that system \eqref{eq1a} admits a solution $u \in C\left([0, T] ;{H}^{1}(\mathbb{R}^{3})\right)$ satisfying
	$u(x, T)=0$ in $\mathbb{R}^3$?
	
	\vspace{0.2cm}

	Control properties of Schr\"{o}dinger equations have received much attention in the last decades. For example, regarding control issues, one may see \cite{KoLo,Miller,Phung,RaTaTeTu} and the references therein. As for Carleman estimates and applications to inverse problems, we cite \cite{BaPu,CaGa,CaCriGa,LaTriZhang, MeOsRo, YuYa} and the references therein. An excellent review of the contributions up to 2003 is in \cite{Zuazua}.
	
	Let us detail some recent results. The results due to Illner \textit{et al.} \cite{IlLaTe,IlLaTe1} considered internal controllability of the nonlinear Schr\"{o}dinger equation posed on a finite interval $(-\pi,\pi)$
	\begin{equation}\label{IlLaTe}
		\begin{cases}
			i\partial_tv+\partial_{xx}v+\lambda|v|^2v=f(x,t), \quad x\in(-\pi,\pi) ,\\
			v(-\pi,t)=v(\pi,t),\text{   }\text{   }\partial_xv(-\pi,t)=\partial_xv(\pi,t), 
		\end{cases}
	\end{equation}
	where the forcing function $f=f(x,t)$, supported in a sub-interval of $(-\pi,\pi)$, is a control input. They showed that system \eqref{IlLaTe} is locally exactly controllable in the space 
	\[
	H^1_p(-\pi,\pi):=\{v\in H^1(-\pi,\pi):v(-\pi)=v(\pi)\}.
	\]
	Later, Lange and Teismann \cite{LaTe} considered the internal control of the same nonlinear Schr\"{o}dinger equation in \eqref{IlLaTe} posed on a finite interval but with the homogeneous Dirichlet boundary conditions
	$$
	v(-\pi,t)=v(\pi,t)=0, 
	$$
	and showed that this is locally exactly controllable in the space $H^1_0(-\pi,\pi)$ around a special ground state of the equation (see also \cite{Laurent} for internal controllability of the nonlinear Schr\"{o}dinger equation posed on a finite interval).
	
	In  \cite{RoZhaJDE}, Rosier and Zhang considered the nonlinear Schr\"{o}dinger equation
	\begin{equation}
		\label{s1_R}
		i\partial_tu+\Delta u +\lambda|u|^2u=0 
	\end{equation}
	posed on a bounded domain $K$ in $\mathbb{R}^n$ with both Dirichlet boundary conditions and Neumann boundary conditions. They showed that if either 
	\[
	s>\frac{n}{2} 
	\]
	or
	\[
	0\leq s<\frac{n}{2},\text{  }\text{ with } 1\leq n<2+2s 
	\]
	or
	\[
	s=0,1\text{  }\text{ with } n=2,
	\]
	then both systems, with Dirichlet and Neumann conditions, when the control inputs are acting on the whole boundary of $K$, are locally exactly controllable in the classical Sobolev space $H^s(K)$ around any smooth solution of the Schr\"{o}dinger equation.
	
	In \cite{RoZhaMMM}, the authors extend the results of Rosier and Zhang \cite{RoZhaSIAM}. More precisely, they assume that the spatial variable belongs to the rectangle
	\[
	\mathcal{R}=(0,l_1)\times\cdot\cdot\cdot\times(0,l_n)
	\]
	and investigate the control properties of the semi-linear Schr\"{o}dinger equation 
	\[
	i\partial_tu+\Delta u+\lambda|u|^{\alpha}u=ia(x)h(x,t), \text{  }\text{  }x\in\mathcal{R},\text{  }\text{  }t\in(0,T),
	\]
	where $\lambda\in\mathbb{R}$ and $\alpha\in 2\mathbb{N}^{\ast}$ by combining new linear controllability results in the spaces $H^s(\mathcal{R})$ with Bourgain analysis. In this case, the geometric control condition is not required (see \cite{RoZhaMMM} for more details). It is important to note that these cases are studied in the subcritical case. 
	
	Finally,  considering a $2d$-compact Riemann manifold $M$ without boundary, Dehman \textit{et al.} \cite{DeGeLe} studied internal control and stabilization of nonlinear Schr\"{o}dinger equations
	\begin{equation*}
		i\partial_tw+\Delta w-|w|^2w=f(x,t),\text{  }x\in M.
	\end{equation*}
	They showed, in particular, that the system is semi-globally exactly controllable and semi-globally exponentially stabilizable in the space $H^1(M)$, assuming that both the geometric control condition and a unique continuation condition are satisfied (see \cite{DeGeLe} for more details). The work \cite{CaPa} extended this one: The authors studied global controllability and stabilization properties for the fractional Schr\"odinger equation on $d-$dimensional compact Riemannian manifolds without boundary $(M, g)$.  They used microlocal analysis to show the propagation of regularity, which, together with the geometric control condition (GCC) and a unique continuation property, allowed them to prove global control results. 
	
	We conclude by mentioning a recent work. In a very interesting work \cite{Taufer}, Ta\"ufer demonstrated that the controllability of the linear Schrödinger equation in $\mathbb{R}^d$ holds for any time $T > 0$ with internal control supported on nonempty, periodic, open sets. This result, in particular, showed that the controllability of the linear Schrödinger equation in full space extends to a strictly larger class of control supports than that of the wave equation.

	\subsection{Main results} As discussed above, most of the available results concern subcritical regimes (with respect to the regularity of the state space), for Schr\"odinger equations posed on bounded domains or compact manifolds, with internal or boundary controls. However, in the energy-critical case in the whole space, namely, for system \eqref{eq1a} in $\mathbb{R}^{3}$, the internal control problem remains open. 
	
	Motivated by the ideas contained in \cite{RoZhaJDE},  let us present the first answer to the control problem stated at the beginning of the section. To do this, consider the control system 
	\begin{equation}\label{eq1}
		\left\{
		\begin{array}{lr}
			i\partial_{t} u + \Delta u  \pm |u|^{4}u =\varphi(x)h(x,t) , \ \mbox{on }[0,+\infty)\times\mathbb{R}^{3},\\
			u(x,0) = u_{0} (x) \in {H}^{1}(\mathbb{R}^{3}),
		\end{array}
		\right.
	\end{equation}
	where  the function $\varphi \in C^{\infty}(\mathbb{R}^3,[0,1])$ is such that
	\begin{equation} \label{varphi} 
		\varphi(x)=0 \quad \mbox{ if } \qquad |x| \leq R, \mbox{ and } \varphi(x) =1 \quad\mbox{ if } \qquad|x| \geq R+1.
	\end{equation}
	 %and satisfies
	%\begin{equation}
	%	\varphi(x) = 	\left\{
		%\begin{array}{lr}
		%	0, \ if \ |x| \leq R,  \\
		%	1, \ if \ |x| \geq R+1,
	%	\end{array}
	%	\right.
%	\end{equation}
for a fixed $R\geq1$.  Our result below gives a first answer in this direction. 
	
	\begin{theorem}[Local controllability]\label{main} Let $T>0$ be given. There exists $\delta > 0$ such that for any $u_{0}\in H^{1}(\mathbb{R}^{3})$ satisfying
		$$\|u_{0}\|_{H^{1}} \leq \delta,$$ one can find $h(x,t) \in C([0,T];H^{1}(\mathbb{R}^{3}))$ such that system \eqref{eq1} admits a solution $u \in C ([0,T] ; {H}^{1}(\mathbb{R}^{3}))$ satisfying $u(T) = 0.$
	\end{theorem}
	In our arguments, Strichartz-type inequalities play a fundamental role in establishing the well-posedness of the system \eqref{eq1}. In addition, as a first step, we have used a unique continuation property for the linear Schr\"odinger equation in $\mathbb{R}^{3}$ (Lemma \ref{UCP} below), which follows from the Paley--Wiener--Schwartz theorem. 
	\begin{theorem}\label{lin}
		For every $T>0$ and every $u_{0}\in H^{1}(\mathbb{R}^{3})$, there exists a control $h\in C([0,T];H^{1}(\mathbb{R}^{3}))$ such that the unique solution $u\in C([0,T];H^{1}(\mathbb{R}^{3}))$ of the linear system associated to \eqref{eq1}, namely,
			\begin{equation*}
			\left\{
			\begin{array}{lr}
				i\partial_{t} u + \Delta u =\varphi(x)h(x,t) , \ \mbox{on }(0,T)\times\mathbb{R}^{3},\\
				u(x,0) = u_{0} (x) \in {H}^{1}(\mathbb{R}^{3}),
			\end{array}
			\right.
		\end{equation*}
		 satisfies $u(\cdot,T) = 0$. Moreover, the control force $\varphi h$ acting on the system is supported in $[0,T]\times\{x\in\mathbb{R}^{3}:\ |x|\geq R\}$.
	\end{theorem}
	\begin{remark}
		The control constructed in the proof of Theorem \ref{lin} is of the form $h(t)=\Lambda^{-1}(\varphi v(t))$, where $\Lambda^{-1}$ is a nonlocal operator (see \eqref{defiso}) and $v$ is a solution of the adjoint system. Hence $h$ itself is, in general, not compactly supported. What is localized is the forcing term $\varphi h$, which vanishes in $B_{R}(0)$ by \eqref{varphi}. The same holds for the control given by Theorem \ref{main}.
	\end{remark}
	With Theorem \ref{lin} in hand, a perturbation argument ensures that we can get a local controllability result
	for the critical Schrödinger equation \eqref{eq1}, so giving Theorem \ref{main}.

	\begin{remark}Let us point out the contributions of this work with respect to \cite{RoZhaJDE,KeMe,Zuazua1}.
\begin{itemize}
\item \emph{Cauchy problem.} The small-data theory for the energy-critical NLS in $H^{1}(\mathbb{R}^{3})$ is classical \cite{cazenave1,cazenave_book}, and our fixed-point argument in Strichartz spaces follows \cite{cazenave1,KeMe}. What is needed for the control problem, and is proved in Theorem \ref{t2.4}, is a version with a forcing term $g\in L^{1}(0,T;H^{1}(\mathbb{R}^{3}))$, posed on an arbitrary fixed interval $[0,T]$, with constants independent of $T$ and with Lipschitz dependence on the data. This is what makes the nonlinear control map of Subsection \ref{sub32} well defined on the whole interval $[0,T]$.
\item \emph{Linear control.} The observability inequality in $H^{-1}(\mathbb{R}^{3})$ follows the compactness--uniqueness strategy of \cite{RoZhaJDE}. We give a self-contained unique continuation argument (Lemma \ref{UCP}), adapted to the whole space $\mathbb{R}^{3}$ and to solutions in $C([0,T];H^{-1}(\mathbb{R}^{3}))$.
\item \emph{Nonlinear control.} Compared with \cite{RoZhaJDE}, whose three-dimensional \(H^1\) result concerns the energy-subcritical cubic equation, we deal with the energy-critical nonlinearity \(|u|^4u\). We close the perturbation argument of \cite{Zuazua1} in Strichartz spaces at the critical regularity (using, in particular, the endpoint estimate of \cite{KeelTao}), instead of the subcritical arguments used in \cite{RoZhaJDE}. On the other hand, working in the whole space introduces additional difficulties; see Section \ref{Sec4}.
\end{itemize}
	\end{remark}
	\subsection{Structure of this work} We finish our introduction by giving an outline of this work. It is divided as follows: In Section \ref{Sec2}, we give auxiliary results that are important in establishing our control result. Precisely, we present a review of the Cauchy problem for the Schrödinger equation. In Section \ref{Sec3}, we present the proof of our main results. Firstly, we show the null controllability of the linear system associated with \eqref{eq1}, proving Theorem \ref{lin}. Then, through a perturbation argument, we show Theorem \ref{main}. Finally, we discuss some future perspectives in Section \ref{Sec4}.

	%%%%%%%%%%%%%%%%%%%%%%%%%%%% sec 2 %%%%%%%%%%%%%%%%%%%%%%%%%%%

	\section{A review of the Cauchy problem\label{Sec2}}
	\subsection{Small-data well-posedness with a forcing term}  Here,  inspired by the work of Kenig and Merle \cite{KeMe}, we are interested in studying the ${H}^{1}$ critical %defocusing,%
	 Cauchy problem for the nonlinear Schr\"{o}dinger equation (C-NLS) with a forcing term,
	\begin{equation} \label{10}
		\left\{
		\begin{array}{lr}
			i\partial_{t} u + \Delta u +N(u) = g, \ \mbox{on }  [0,T] \times \mathbb{R}^{3}, \qquad N(u):=\pm|u|^{4}u,\\
			u(t_{0}) = u_{t_{0}} \in {H}^{1}(\mathbb{R}^{3}),
		\end{array}
		\right.
	\end{equation}
	where $T>0$, $t_{0}\in[0,T]$ and $g \in L^{1}(0,T;H^{1}(\mathbb{R}^{3}))$. Let us first present some definitions used throughout the paper.
	\begin{definition} A pair $(q,r)$ is called $L^{2}$-admissible if $r \in [2,6]$ and q satisfies
		\begin{equation*}
			\frac{2}{q} + \frac{3}{r} = \frac{3}{2}.
		\end{equation*}
	\end{definition}
	With this definition in hand, we present two results that are paramount to proving that the Cauchy problem \eqref{10} is well-posed. The first one is \textit{Strichartz estimates}, and the second one is a standard Sobolev embedding. These results can be found in \cite{cazenave_book, KeelTao}. Throughout this paper, for an interval $I\subset\mathbb{R}$, we will use the notation $L^{q}_{t}L^{r}_{x}$ to denote mixed space-time spaces $L^{q}(I;L^{r}(\mathbb{R}^{3}))$. For $s\in\mathbb{R}$, the Sobolev space $H^{s}(\mathbb{R}^{3})$ is endowed with the norm $\|f\|_{H^{s}}:=\|(I-\Delta)^{s/2}f\|_{L^{2}(\mathbb{R}^{3})}$; in particular, $$\|f\|^{2}_{H^{1}}=\|f\|^{2}_{L^{2}}+\|\nabla f\|^{2}_{L^{2}},$$ and the norm of $H^{-s}(\mathbb{R}^{3})$ coincides with the dual norm of $H^{s}(\mathbb{R}^{3})$ with respect to the $L^{2}$ pairing. With these conventions, $e^{it\Delta}$ is unitary on every $H^{s}(\mathbb{R}^{3})$, and $I-\Delta$ is an isometry from $H^{s}(\mathbb{R}^{3})$ onto $H^{s-2}(\mathbb{R}^{3})$. 
	
	For $p\in[1,\infty]$ we denote by $p'$ its H\"older conjugate exponent, defined by $\frac1p+\frac1{p'}=1$. If $(q,r)$ is an $L^{2}$-admissible pair, the pair $(q',r')$ is called its \emph{dual pair}.

	\begin{lemma}[Strichartz estimates \cite{cazenave_book,KeelTao}]\label{strichartz}
		Let $(q,r)$ and $(\tilde q,\tilde r)$ be two (possibly different) $L^{2}$-admissible pairs, let $I\subset\mathbb{R}$ be an interval and $t_{0}\in \overline{I}$. There exists a constant $C_{S}>0$, depending only on $(q,r)$ and $(\tilde q,\tilde r)$ (in particular, independent of $I$ and $t_{0}$), such that
		\begin{align}
			\|e^{it\Delta}h\|_{L^{q}_{t}L^{r}_{x}} &\leq C_{S}\|h\|_{L^{2}_{x}}, \label{item_i}\\
			\Big\|\int_{t_{0}}^{t}e^{i(t-\tau)\Delta}F(\tau)\, d\tau\Big\|_{L^{q}_{t}L^{r}_{x}}
			&\leq C_{S}\|F\|_{L^{\tilde q'}_{t}L^{\tilde r'}_{x}} . \label{item_ii}
		\end{align}
	\end{lemma}
	\begin{remark}\label{pairs}
		We shall use the following $L^{2}$-admissible pairs, with their dual pairs:
		\begin{itemize}
			\item[i)] $(\infty,2)$, with dual pair $(1,2)$;
			\item[ii)] $\big(10,\frac{30}{13}\big)$, with dual pair $\big(\frac{10}{9},\frac{30}{17}\big)$;
			\item[iii)] the endpoint pair $(2,6)$ \cite{KeelTao}, with dual pair $\big(2,\frac65\big)$.
		\end{itemize}
		In all the nonlinear estimates below, \eqref{item_ii} is used with $(q,r)\in\{(\infty,2),(10,\frac{30}{13})\}$ on the left-hand side and $(\tilde q,\tilde r)=(2,6)$ on the right-hand side, so that the nonlinear terms are estimated in $L^{2}_{t}L^{6/5}_{x}$; for the forcing term we use $(\tilde q,\tilde r)=(\infty,2)$, so that it is estimated in $L^{1}_{t}L^{2}_{x}$. Since only finitely many combinations of pairs are used, we denote by $C_{S}$ the maximum of the corresponding constants in Lemma \ref{strichartz}; without loss of generality, we also assume $C_{S}\geq1$.
	\end{remark}
	
\begin{lemma}[Sobolev Embedding]\label{sobolev}Let $I$ be an interval and $v \in L^{10}(I;W^{1,30/13}(\mathbb{R}^{3}))$. Then $v\in L^{10}(I;L^{10}(\mathbb{R}^{3}))$ and
	$$\|v\|_{L^{10}_{t}L^{10}_{x}} \leq C_{Sob} \|\nabla v\|_{L^{10}_{t}L^{\frac{30}{13}}_{x}}.$$
\end{lemma}
Indeed, for a.e. $t\in I$, $v(t)\in W^{1,30/13}(\mathbb{R}^{3})$, and the Sobolev inequality $$\|f\|_{L^{10}}\leq C_{Sob}\|\nabla f\|_{L^{30/13}}$$ (note that $\frac{1}{10}=\frac{13}{30}-\frac13$) holds for every $f\in W^{1,30/13}(\mathbb{R}^{3})$; it suffices to take the $L^{10}$ norm in time. In particular, the lemma applies to every $u$ with $\|u\|_{X(I)}<\infty$, where $\|\cdot\|_{X(I)}$ is defined below.
For an interval $I$, define the norms $S(I)$ and $Z(I)$ by
$$\|u\|_{S(I)} = \|u\|_{L^{10}(I;L^{10}(\mathbb{R}^{3}))} \ \mbox{and} \ \  \|u\|_{Z(I)} = \|u\|_{L^{10}(I;L^{\frac{30}{13}}(\mathbb{R}^{3}))},$$
and set
$$\|u\|_{X(I)}:=\|u\|_{L^{\infty}(I;H^{1}(\mathbb{R}^{3}))}+\|u\|_{Z(I)}+\|\nabla u\|_{Z(I)}.$$
By Lemma \ref{sobolev}, $$\|u\|_{S(I)}\leq C_{Sob}\|\nabla u\|_{Z(I)}\leq C_{Sob}\|u\|_{X(I)}.$$ 
The following theorem gives us the solution to problem \eqref{10}. The key point is that the solution exists on the whole interval $[0,T]$, provided that the initial datum and the forcing term are small, and that all constants are independent of $T$.

\begin{theorem} \label{t2.4}
	
	There exist $\eta_{0}>0$ and $C_{0}\geq1$, depending only on the constants $C_{S}$ of Lemma \ref{strichartz} and $C_{Sob}$ of Lemma \ref{sobolev} (in particular, independent of $T$), such that the following holds. Let $T>0$, $I=[0,T]$, $t_{0}\in I$, $u_{t_{0}}\in H^{1}(\mathbb{R}^{3})$ and $g\in L^{1}(I;H^{1}(\mathbb{R}^{3}))$ with
	\[
	\eta:=\|u_{t_{0}}\|_{H^{1}}+\|g\|_{L^{1}(I;H^{1})}\leq\eta_{0}.
	\]
	Then problem \eqref{10} has a unique solution in the class $$C(I;H^{1}(\mathbb{R}^{3}))\cap L^{10}(I;W^{1,30/13}(\mathbb{R}^{3})),$$ and this solution satisfies $\|u\|_{X(I)}\leq 2C_{0}\eta$. Moreover, if $u_{1},u_{2}$ are the solutions associated with $(u^{1}_{t_{0}},g_{1})$ and $(u^{2}_{t_{0}},g_{2})$, both satisfying the smallness condition above, then
	\begin{equation}\label{lip}
		\|u_{1}-u_{2}\|_{X(I)}\leq 2C_{0}\Big(\|u^{1}_{t_{0}}-u^{2}_{t_{0}}\|_{H^{1}}+\|g_{1}-g_{2}\|_{L^{1}(I;H^{1})}\Big).
	\end{equation}
\end{theorem}
\begin{proof}
	We follow the ideas of \cite{cazenave1,KeMe}. Problem \eqref{10} is equivalent to the integral equation (Duhamel's formula)
	\[
	u(t)=\mathcal{T}(u)(t):=e^{i(t-t_{0})\Delta}u_{t_{0}}+i\int_{t_{0}}^{t}e^{i(t-\tau)\Delta}\big(N(u)-g\big)(\tau)\,d\tau .
	\]
	By Lemma \ref{strichartz}, with the pairs indicated in Remark \ref{pairs}, each of the three components of the norm $\|\cdot\|_{X(I)}$ of $\mathcal{T}(u)$ is bounded by
	\[
	C_{S}\Big(\|u_{t_{0}}\|_{H^{1}}+2\|g\|_{L^{1}(I;H^{1})}+\|N(u)\|_{L^{2}_{t}L^{6/5}_{x}}+\|\nabla N(u)\|_{L^{2}_{t}L^{6/5}_{x}}\Big).
	\]
	By H\"older's inequality ($\frac{4}{10}+\frac{1}{10}=\frac12$ in time and $\frac{4}{10}+\frac{13}{30}=\frac56$ in space),
	\begin{equation}\label{holder-N}
		\|N(u)\|_{L^{2}_{t}L^{6/5}_{x}}\leq\|u\|_{S(I)}^{4}\|u\|_{Z(I)},\qquad
		\|\nabla N(u)\|_{L^{2}_{t}L^{6/5}_{x}}\leq5\|u\|_{S(I)}^{4}\|\nabla u\|_{Z(I)},
	\end{equation}
	and therefore, by Lemma \ref{sobolev},
	\[
	\|\mathcal{T}(u)\|_{X(I)}\leq C_{0}\,\eta+C_{1}\|u\|_{X(I)}^{5},\qquad C_{0}:=6C_{S},\quad C_{1}:=18C_{S}C_{Sob}^{4}.
	\]
	Similarly, using the pointwise bounds
	$$|N(u)-N(v)|\leq 5(|u|^{4}+|v|^{4})|u-v|$$ and
	\[
	|\nabla N(u)-\nabla N(v)|\leq C(|u|^{4}+|v|^{4})|\nabla(u-v)|+C(|u|^{3}+|v|^{3})(|\nabla u|+|\nabla v|)|u-v|,
	\]
	H\"older's inequality (for the last term, $\frac{3}{10}+\frac{1}{10}+\frac{1}{10}=\frac12$ in time and $\frac{3}{10}+\frac{13}{30}+\frac{1}{10}=\frac56$ in space) and Lemma \ref{sobolev} applied to $u$, $v$ and $u-v$, we get
	\[
	\|\mathcal{T}(u)-\mathcal{T}(v)\|_{X(I)}\leq C_{2}\big(\|u\|_{X(I)}^{4}+\|v\|_{X(I)}^{4}\big)\|u-v\|_{X(I)},
	\]
	where $C_{2}$ depends only on $C_{S}$ and $C_{Sob}$. Let
	$$B=\{u\in C(I;H^{1}(\mathbb{R}^{3})):\|u\|_{X(I)}\leq 2C_{0}\eta\},$$ which is a complete metric space for the distance induced by $\|\cdot\|_{X(I)}$. Choose $\eta_{0}>0$ such that
	\[
	C_{1}(2C_{0}\eta_{0})^{4}\leq\tfrac12
	\qquad\text{and}\qquad
	2C_{2}(2C_{0}\eta_{0})^{4}\leq\tfrac12 .
	\]
	Then, for $\eta\leq\eta_{0}$, $$\|\mathcal{T}(u)\|_{X(I)}\leq C_{0}\eta+\frac12\cdot 2C_{0}\eta= 2C_{0}\eta$$ for $u\in B$, so $\mathcal{T}(B)\subset B$, and $\mathcal{T}$ is a contraction on $B$ with constant $\frac12$. The Banach fixed-point theorem gives existence and uniqueness of the solution in $B$. Here we used that $\mathcal{T}$ maps $B$ into $C(I;H^{1}(\mathbb{R}^{3}))$. Indeed, if $F\in L^{2}(I;L^{6/5}(\mathbb{R}^{3}))$ or $F\in L^{1}(I;L^{2}(\mathbb{R}^{3}))$, the function $t\mapsto\int_{t_{0}}^{t}e^{i(t-\tau)\Delta}F(\tau)\,d\tau$ belongs to $C(I;L^{2}(\mathbb{R}^{3}))$: this is clear when $F$ is smooth and compactly supported in $I\times\mathbb{R}^{3}$, and it follows in general by density, since \eqref{item_ii} with $(q,r)=(\infty,2)$ gives a bound in $L^{\infty}(I;L^{2}(\mathbb{R}^{3}))$ which is uniform with respect to the corresponding norm of $F$. Applying this to $F$ and $\nabla F$, with $F=N(u)$ and $F=g$, we obtain $\mathcal{T}(u)\in C(I;H^{1}(\mathbb{R}^{3}))$. 
	
	Let us now prove uniqueness in the class of solutions $\tilde u\in C(I;H^{1}(\mathbb{R}^{3}))$ with $\tilde u,\nabla\tilde u\in Z(I)$ (see also \cite[Section 4.2]{cazenave_book} and \cite{cazenave1}). For a subinterval $J\subset I$, set $$\|w\|_{Y(J)}:=\|w\|_{Z(J)}+\|\nabla w\|_{Z(J)}.$$ By Lemma \ref{sobolev}, the nonlinear estimates above only involve the norm $\|\cdot\|_{Y(J)}$, so that, if $u$ and $\tilde u$ solve \eqref{10} on $J=[a,b]$ with $u(a)=\tilde u(a)$, then
\[
\|u-\tilde u\|_{Y(J)}\leq C_{2}\big(\|u\|_{Y(J)}^{4}+\|\tilde u\|_{Y(J)}^{4}\big)\|u-\tilde u\|_{Y(J)} .
\]
Let $$t^{*}:=\sup\{t\in[t_{0},T]:\ u=\tilde u \text{ on } [t_{0},t]\}$$ and assume, by contradiction, that $t^{*}<T$. By continuity, $u(t^{*})=\tilde u(t^{*})$. Since $\|u\|_{Y(I)}$ and $\|\tilde u\|_{Y(I)}$ are finite, the absolute continuity of the integral provides $\varepsilon>0$ such that, with $J=[t^{*},t^{*}+\varepsilon]\subset I$, one has $$C_{2}(\|u\|_{Y(J)}^{4}+\|\tilde u\|_{Y(J)}^{4})<1.$$ Hence $u=\tilde u$ on $J$, which contradicts the definition of $t^{*}$. Therefore $u=\tilde u$ on $[t_{0},T]$, and the same argument applies on $[0,t_{0}]$.

	Finally, if $u_{1},u_{2}$ are as in the statement, writing $u_{j}=\mathcal{T}_{j}(u_{j})$ and using the above estimates,
	\[
	\|u_{1}-u_{2}\|_{X(I)}\leq C_{0}\Big(\|u^{1}_{t_{0}}-u^{2}_{t_{0}}\|_{H^{1}}+\|g_{1}-g_{2}\|_{L^{1}(I;H^{1})}\Big)+\tfrac12\|u_{1}-u_{2}\|_{X(I)},
	\]
	which gives \eqref{lip}.
\end{proof}

\begin{remark}\label{rem-length}
	No restriction on the length of the interval $I$ is required in Theorem \ref{t2.4}: the constants $\eta_{0}$ and $C_{0}$ are independent of $T$, and $T$ enters only through the size of the forcing term, via
	\[
	\|g\|_{L^{1}(0,T;H^{1})}\leq T^{1/2}\|g\|_{L^{2}(0,T;H^{1})}\leq T\|g\|_{L^{\infty}(0,T;H^{1})}.
	\]
	Since the Schr\"odinger equation is reversible in time, the case $t_{0}=T$ corresponds to the backward problem with final datum $u(T)=u_{T}$; the case $u_{T}=0$ is the one used in Subsection 3.2.
\end{remark}

		%%%%%%%%%%%%%%%%%%%%%%%%%%%% sec 3 %%%%%%%%%%%%%%%%%%%%%%%%%%%
		
		\section{Controllability for the critical nonlinear Schr\"{o}dinger equation}\label{Sec3}
		In this section, we prove Theorem \ref{main}. We first establish the null controllability of the linear system
		\begin{equation}\label{linear}
			\left\{
			\begin{array}{lr}
				i\partial_{t}u + \Delta u = \varphi(x) h(x,t), \  \mbox{on  } (0,T) \times \mathbb{R}^{3}, \\
				u(0) = u_{0},
			\end{array}
			\right.
		\end{equation}
		where $\varphi=\varphi(x)$ is defined by \eqref{varphi}, using a duality strategy, the so-called ``Hilbert Uniqueness Method" \cite{Lions}, which reduces the null controllability of \eqref{linear} to an observability inequality for the solutions of the adjoint system \eqref{c2} below. Then, a perturbation argument gives the result for the nonlinear system \eqref{eq1}.
		
		\subsection{Linear Schr\"{o}dinger equation: Null controllability}
		In this section, inspired by ideas of Rosier and Zhang \cite{RoZhaJDE}, we will show Theorem \ref{lin}. Note that the null controllability of system \eqref{linear} follows from the observability inequality, namely, 
		\begin{equation}\label{c1}
			\|v_{0}\|^{2}_{H^{-1}} \leq C \int_{0}^{T} \|\varphi v(t)\|^{2}_{H^{-1}} \ dt,
		\end{equation}
		where $v(x,t)$ is a solution to the adjoint system associated to \eqref{linear}
		\begin{equation} \label{c2} 
			\left\{
			\begin{array}{lr}
				i\partial_{t} v + \Delta v = 0 \ \mbox{on} \ \mathbb{R}\times \mathbb{R}^{3} , \\
				v(0) = v_{0}  \in H^{-1}(\mathbb{R}^{3}) .
			\end{array}
			\right.
		\end{equation}
		
		We first state the unique continuation property that will be used in the proof of the observability inequality. It replaces the use of smoothing properties and of Carleman estimates on bounded domains, and applies directly in the whole space $\mathbb{R}^{3}$ and to solutions with values in negative Sobolev spaces.
		\begin{lemma}\label{UCP}
			Let $s\in\mathbb{R}$, $v_{0}\in H^{s}(\mathbb{R}^{3})$ and $v(t)=e^{it\Delta}v_{0}$. Assume that there exist $t_{1}\neq t_{2}$ and $\rho>0$ such that $\supp v(t_{1})\cup\supp v(t_{2})\subset\overline{B_{\rho}(0)}$. Then $v_{0}=0$.
		\end{lemma}
		\begin{proof}
			By the group property, we may assume $t_{1}=0$. Set $s_{0}:=t_{2}\neq0$ and $f:=v(0)$, which is a compactly supported distribution. For $\psi\in\mathcal{S}(\mathbb{R}^{3})$,
			\[
			e^{is_{0}\Delta}\psi(x)=(4\pi i s_{0})^{-3/2}\,e^{i\frac{|x|^{2}}{4s_{0}}}\int_{\mathbb{R}^{3}}
			e^{-i\frac{x\cdot y}{2s_{0}}}\,e^{i\frac{|y|^{2}}{4s_{0}}}\psi(y)\,dy ,
			\]
			that is, $e^{is_{0}\Delta}=c_{s_{0}}\,\mathcal{M}\,\mathcal{D}\,\mathcal{F}\,\mathcal{M}$, where $\mathcal{M}$ denotes the multiplication by $e^{i|x|^{2}/(4s_{0})}$, $\mathcal{F}$ the Fourier transform, $\mathcal{D}g(x)=g(x/(2s_{0}))$ and $c_{s_{0}}=(4\pi i s_{0})^{-3/2}$. All these operators are continuous on $\mathcal{S}'(\mathbb{R}^{3})$ (the multiplier $e^{i|x|^{2}/(4s_{0})}$ is smooth with polynomially bounded derivatives), so this identity extends to $\mathcal{S}'(\mathbb{R}^{3})$. The distribution $g:=\mathcal{M}f$ has compact support, hence, by the Paley--Wiener--Schwartz theorem, $\widehat{g}$ is the restriction to $\mathbb{R}^{3}$ of an entire function. Since $v(s_{0})=c_{s_{0}}\mathcal{M}\mathcal{D}\widehat g$ vanishes in $\{|x|>\rho\}$, the function $\widehat g$ vanishes on the nonempty open set $\{|\xi|>\rho/(2|s_{0}|)\}$. By analyticity, $\widehat g\equiv0$, so that $g=0$ and $f=v_{0}=0$.
		\end{proof}
		
		The observability inequality \eqref{c1} is given by the following result. 
		\begin{proposition}\label{obs2} Let $\varphi$ be a $C^{\infty}$ real function on $\mathbb{R}^{3}$ as in \eqref{varphi}. Then, for every $T > 0$, there exists a constant $C = C(T,\varphi) > 0$ such that inequality \eqref{c1}	holds for every solution $v(x,t)$ of system \eqref{c2}.
		\end{proposition}
		\begin{proof}
			We will split the proof into several steps.
			
			\vspace{0.2cm}
			
			\noindent\textbf{Step 1:} \textit{$H^{1}$--observability}.
			
			\begin{lemma}\label{obs1} Consider the system
				\begin{equation}  \label{c25}
					\left\{
					\begin{array}{lr}
						i\partial_{t} w + \Delta w = 0, \  \mbox{on  } (0,T) \times \mathbb{R}^{3}, \\
						w(0) = w_{0} \in H^{1}(\mathbb{R}^{3}).
					\end{array}
					\right.
				\end{equation}
				There exists a constant $C = C(T,R) > 0$ such that for each $w_{0} \in H^{1}(\mathbb{R}^{3})$, the solution $w(t)$ of \eqref{c25} satisfies
				\begin{equation} \label{c26}
					\|w_{0}\|^{2}_{H^{1}(\mathbb{R}^{3})} \leq C \int_{0}^{T} \|\varphi w(t)\|^{2}_{H^{1}(\mathbb{R}^{3})} \ dt .
				\end{equation}
			\end{lemma}
			\begin{proof}
				Let $q \in C_{c}^{\infty}(\mathbb{R}^{3};\mathbb{R}^3)$ such that
				\begin{equation*} 
					q(x) = 	\left\{
					\begin{array}{lr}
						x, \ if \ |x| \leq R + 2  , \\
						0, \ if \ |x| \geq R+3 .
					\end{array}
					\right.
				\end{equation*}
				Multiplying the equation in \eqref{c25} by $q\cdot \nabla \overline{w} + \frac{1}{2} \overline{w}(\mathrm{div}_{x}q)$, taking the real part and integrating by parts gives
				\begin{equation} \label{c27}
					\begin{split}
						\frac{1}{2} \im\int_{\mathbb{R}^{3}} (wq \cdot \nabla \overline{w}) \ dx \Bigg\arrowvert_{0}^{T}& + \frac{1}{2} \re \int_{0}^{T} \int_{\mathbb{R}^{3}} w \nabla (\mathrm{div}_{x}q) \cdot \nabla \overline{w} \ dxdt\\& + \re \int_{0}^{T} \int_{\mathbb{R}^{3}} \sum_{j,k =1}^{3} \Big(\frac{\partial q_{k}}{\partial x_{j}}\frac{\partial \overline{w}}{\partial x_{k}}\frac{\partial w}{\partial x_{j}}\Big) \ dxdt = 0,
					\end{split}
				\end{equation}
				where we have used the fact that the function $q(x)$ has compact support. Identity \eqref{c27} is first obtained for $w_{0}\in H^{3}(\mathbb{R}^{3})$, for which $$w\in C([0,T];H^{3}(\mathbb{R}^{3}))\cap C^{1}([0,T];H^{1}(\mathbb{R}^{3}))$$ and all the integrations by parts are legitimate; it then extends to every $w_{0}\in H^{1}(\mathbb{R}^{3})$ by density, since each term in \eqref{c27} is continuous with respect to $w_{0}$ in the $H^{1}$ topology (recall that $q$ is smooth and compactly supported and that $\|w(t)\|_{H^{1}}=\|w_{0}\|_{H^{1}}$).
				Notice that there is conservation of energy in $H^{1}(\mathbb{R}^{3})$
				\begin{equation*} %\label{c28}
					\|w(t)\|^{2}_{H^{1}(\mathbb{R}^{3})} = \|w(0)\|^{2}_{H^{1}(\mathbb{R}^{3})}=\|w_0\|^{2}_{H^{1}(\mathbb{R}^{3})}, \quad t \in [0,T].
				\end{equation*}
				It follows from \eqref{c27} that
				\begin{eqnarray*}
					\int_{0}^{T} \int_{B_{R+2}(0)} |\nabla w|^{2} \ dx dt & \leq & C_{\varepsilon} \Bigg(	\int_{0}^{T} \int_{B_{R+3}(0)\backslash B_{R+2}(0)} |\nabla w|^{2} \ dx dt + \|w_0\|^2_{L^2(\mathbb{R}^{3})} \Bigg) \\
					&    &  \mbox{    } + \varepsilon \int_{0}^{T} \|w(t)\|^{2}_{H^{1}(\mathbb{R}^{3})} \ dt,
				\end{eqnarray*}
				for any $\varepsilon>0$ and some constant $C_{\varepsilon} >0$. We also have
				$$\|w(t)\|^{2}_{H^{1}(\mathbb{R}^{3})} \leq C_R \bigg(\int_{B_{R+2}(0)} |\nabla w|^{2} \ dx + \|\varphi w(t)\|^{2}_{H^{1}(\mathbb{R}^{3})} \bigg)
				$$
				Indeed, observe that
				\begin{eqnarray*}
					\|\nabla w(t)\|^{2}_{L^2(\mathbb{R}^{3})} & = & 	\| \nabla w(t)\|^{2}_{L^2(B_{R+1}(0))} + 	\|\nabla w(t)\|^{2}_{L^2(\mathbb{R}^{3}\backslash B_{R+1}(0))} \\
					& \leq & 	\| \nabla w(t)\|^{2}_{L^2(B_{R+2}(0))}  + 	\|\varphi w(t)\|^{2}_{H^1(\mathbb{R}^{3})}
				\end{eqnarray*}
				and that 
				\begin{eqnarray*}
					\|w(t)\|^{2}_{L^2(\mathbb{R}^{3})} & = & 	\|w(t)\|^{2}_{L^2(B_{R+1}(0))} + 	\|w(t)\|^{2}_{L^2(\mathbb{R}^{3}\backslash B_{R+1}(0))} \\
					& \leq & |B_{R+1}(0)|^{\frac{2}{3}} \|w(t)\|^2_{L^6(B_{R+1}(0))}  + 	\|\varphi w(t)\|^{2}_{L^2(\mathbb{R}^{3}\backslash B_{R+1}(0))} \\
					& \leq & |B_{R+1}(0)|^{\frac{2}{3}} \|w(t)\|^2_{L^6(\mathbb{R}^3)}  + 	\|\varphi w(t)\|^{2}_{H^1(\mathbb{R}^{3})}
				\end{eqnarray*}
				and using the Gagliardo-Nirenberg-Sobolev inequality
				\[ \|w(t)\|_{L^6(\mathbb{R}^3)} \leq C_{\mathrm{GNS}} \|\nabla w(t)\|_{L^2(\mathbb{R}^3)} \]  
				shows the claim. Moreover, we also obtain
				\begin{equation} \label{c29}
					\|w_0\|^2_{H^{1}(\mathbb{R}^{3})} = \frac{1}{T} \int_{0}^{T} 	\|w(t)\|^{2}_{H^{1}(\mathbb{R}^{3})} \ dt \leq C \Bigg( \int_{0}^{T} \|\varphi w(t)\|^{2}_{H^{1}(\mathbb{R}^{3})} \ dt + \|w_0\|^{2}_{L^{2}(\mathbb{R}^{3})}\Bigg).
				\end{equation}
				Indeed we have 
				\begin{eqnarray*} 
					\int_{0}^{T}\|w(t)\|^{2}_{H^{1}(\mathbb{R}^{3})} \ dt  & \leq & C_R \Bigg(  \int_{0}^{T} \|\nabla w(t)\|^{2}_{L^{2}(B_{R+2}(0))} \ dt + 	\int_{0}^{T}\|\varphi w(t)\|^{2}_{H^{1}(\mathbb{R}^{3})} \ dt \Bigg) \\
					& \leq & C_{\varepsilon}C_R  \Bigg(	\int_{0}^{T} \int_{B_{R+3}(0)\backslash B_{R+2}(0)} |\nabla w|^{2} \ dx dt +  \|w_0\|^2_{L^{2}(\mathbb{R}^{3})} \Bigg) \\
					&    &  \mbox{    } + 	C_R \int_{0}^{T}\|\varphi w(t)\|^{2}_{H^{1}(\mathbb{R}^{3})} \ dt 
					+ C_R \varepsilon \int_{0}^{T} \|w(t)\|^{2}_{H^{1}(\mathbb{R}^{3})} \ dt \\
					& \leq & C_{\varepsilon} C_R \|w_0\|^2_{L^2(\mathbb{R}^3)} + (C_{\varepsilon}+1)C_R \int_{0}^{T} \|\varphi w(t)\|^{2}_{H^{1}(\mathbb{R}^{3})} dt \\
					&    &  \mbox{    } +  C_R\varepsilon \int_{0}^{T} \|w(t)\|^{2}_{H^{1}(\mathbb{R}^{3})} \ dt 
				\end{eqnarray*}
				and it suffices to choose $\varepsilon$ small enough to absorb the last term on the right-hand side into the left-hand side. 
				So, it remains to show
				\begin{equation} \label{c30}
					 \|w_0\|^2_{L^{2}(\mathbb{R}^{3})}  \leq c \int_{0}^{T} \|\varphi w(t)\|^{2}_{H^{1}(\mathbb{R}^{3})} dt
				\end{equation}
				to achieve Lemma \ref{obs1}.  
				To this end, let us argue by contradiction; that is, suppose that \eqref{c30} does not hold. If this is the case, there exists a sequence $\{w_{n,0}\}$ in $H^{1}(\mathbb{R}^{3})$ such that the corresponding sequence of solutions $\{w_{n}\}$ of \eqref{c25} satisfies
				\begin{equation} \label{c31}
					1 = 	 \|w_{n,0}\|_{L^{2}(\mathbb{R}^{3})}^2 \geq n \int_{0}^{T} \|\varphi w_{n}(t)\|^{2}_{H^{1}(\mathbb{R}^{3})} dt, \ \ n=1,2,...
				\end{equation}
				Due to inequalities \eqref{c29} and \eqref{c31}, we get 
				$$	\|w_{n,0}\|^{2}_{H^{1}(\mathbb{R}^{3})} = \frac{1}{T} \int_0^T \|w_{n}(t)\|^{2}_{H^{1}(\mathbb{R}^{3})} dt  \leq C \Bigg( \int_{0}^{T} \|\varphi w_{n}(t)\|^{2}_{H^{1}(\mathbb{R}^{3})} \ dt + \|w_{n,0}\|^{2}_{L^{2}(\mathbb{R}^{3})}\Bigg) \leq 2C,$$
				hence the sequence $\{w_{n}(0)= w_{n,0}\}_{n\in\mathbb{N}}$ is bounded in $H^{1}(\mathbb{R}^{3})$. Extracting a subsequence and still denoting it by $\{w_{n,0}\}$, we may assume that
				$$ w_{n,0} \rightharpoonup w_{0} \mbox{  weakly in } H^{1}(\mathbb{R}^{3})$$
				and
				$$ w_{n} \rightharpoonup w \mbox{  weakly in  } L^{2}\big((0,T);H^{1}(\mathbb{R}^{3})\big),$$
				where $w \in C\big([0,T];H^{1}(\mathbb{R}^{3})\big)$ is a solution of \eqref{c25}. By inequality \eqref{c31}, 
				$$ \varphi w_{n} \rightarrow 0 \textrm{ in } L^{2}\big((0,T);H^{1}(\mathbb{R}^{3})\big) \textrm{ strongly.} $$ 
				On the other hand, since $w_{n} \rightharpoonup w$ and the multiplication by $\varphi$ is a bounded linear operator on this space, we have
				$$\varphi w_{n} \rightharpoonup \varphi w \mbox{  weakly in } L^{2}\big((0,T);H^{1}(\mathbb{R}^{3})\big).$$
				By uniqueness of the weak limit, we conclude that  $ \varphi w \equiv 0$ on $(0,T) \times \mathbb{R}^{3}$. Since $$t\mapsto\varphi w(t)\in H^{1}(\mathbb{R}^{3})$$ is continuous, $\varphi w(t)=0$ for every $t\in[0,T]$. As $\varphi\equiv1$ on $\{|x|\geq R+1\}$, this means that $\supp w(t)\subset\overline{B_{R+1}(0)}$ for every $t\in[0,T]$. Applying Lemma \ref{UCP} with $s=1$, $t_{1}=0$ and $t_{2}=T$, we obtain $w_{0}=0$, that is,
				\begin{equation*}
					w \equiv 0 \mbox{  on  } (0,T)\times\mathbb{R}^{3}.
				\end{equation*}
				No smoothing property of $w$ is needed in this argument. 
				Since $$\varphi w_{n} \rightarrow 0$$ strongly in $L^{2}\big((0,T);H^{1}(\mathbb{R}^{3})\big)$, we get 
				\begin{equation} \label{c35}
					w_{n} \rightarrow 0 \mbox{  strongly in } L^{2}\big((0,T);H^{1}(\mathbb{R}^{3}\backslash B_{R+1}(0))\big).
				\end{equation}
				On the other hand, taking into account \eqref{c29}, we have 
				\begin{eqnarray*}
					\frac{1}{T} \int_{0}^{T} \|w_{n}(t)\|^{2}_{H^{1}(B_{R+1}(0))} \ dt %& \leq & 	\int_{0}^{T} \|w_{n}(t)\|^{2}_{H^{1}(\mathbb{R}^{3})} \ dt \\%
					& \leq & C \Big(\int_{0}^{T} \|\varphi w_{n}(t)\|^{2}_{H^{1}(\mathbb{R}^{3})} \ dt + \|w_{n,0}\|^{2}_{L^{2}(\mathbb{R}^{3})}  \Big) \leq 2C,
				\end{eqnarray*}
				and using the equation  \eqref{c25}, 
				\begin{eqnarray*}
					\int_{0}^{T} \|\partial_{t}w_{n}(t)\|^{2}_{H^{-1}(B_{R+1}(0))} \ dt & = &	\int_{0}^{T} \|\Delta w_{n}(t)\|^{2}_{H^{-1}(B_{R+1}(0))} \ dt \\
					%& = & 	\int_{0}^{T} \|(1-\Delta)w_{n}(t) - w_{n}(t)\|^{2}_{H^{-1}(B_{R+1}(0))} \ dt \\
					%& \leq & 	\int_{0}^{T} \|(1-\Delta)w_{n}(t)\|^{2}_{H^{-1}(B_{R+1}(0))}  \ dt 	\\
					%&   & \mbox{   } +2 \int_{0}^{T} \|(1-\Delta)w_{n}(t)\|_{H^{-1}(B_{R+1}(0))}  \|w_{n}(t)\|_{H^{-1}(B_{R+1}(0))} \ dt  \\
					%&   & \mbox{    } + \int_{0}^{T}\|w_{n}(t)\|^{2}_{H^{-1}(B_{R+1}(0))} \ dt\\%
					& \leq & C\int_{0}^{T} \|w_{n}(t)\|^{2}_{H^{1}(B_{R+1}(0))} dt 
				\end{eqnarray*}
				as well as  
				\begin{equation*}
					\begin{split}
						\int_{0}^{T} \|\nabla w_{n}(t)\|^{2}_{H^{-1}(B_{R+1}(0))} \ dt % \leq & C 	\int_{0}^{T} \|\nabla w_{n}(t)\|^{2}_{L^{2}(B_{R+1}(0))} \ dt\\ \leq &C \int_{0}^{T} \|\nabla w_{n}(t)\|^{2}_{L^{2}(\mathbb{R}^{3})} \ dt \\%
						&   \leq C \int_{0}^{T} \| w_{n}(t)\|^{2}_{H^{1}(\mathbb{R}^{3})} \ dt.
					\end{split}
				\end{equation*}
				Therefore, 
				\begin{equation*}
					w_{n} \mbox{  is bounded in } L^{2}\big((0,T);H^{1}( B_{R+1}(0))\big) \cap H^{1}\big((0,T);H^{-1}(B_{R+1}(0))\big).
				\end{equation*}
				Due to Aubin's lemma (see \cite{Simon}) and the convergence \eqref{c35}, we conclude that for a subsequence, still denoted by $\{w_{n}\}$,
				$$w_{n} \rightarrow w = 0 \mbox{  strongly in } L^{2}\big((0,T);L^{2}(\mathbb{R}^{3})\big)$$
				which contradicts \eqref{c31}, since, by the conservation of the $L^{2}$ norm, $\int_{0}^{T}\|w_{n}(t)\|^{2}_{L^{2}(\mathbb{R}^{3})}\,dt=T\|w_{n,0}\|^{2}_{L^{2}(\mathbb{R}^{3})}=T$ for every $n$. So, the estimate \eqref{c26} follows from \eqref{c29} and \eqref{c30}
				showing the lemma. 
			\end{proof}

			\noindent\textbf{Step 2:} \textit{Weak observability inequality.}
			
			\vspace{0.1cm}
			We prove a bound which is weaker than the observability inequality \eqref{c1}.
			\begin{lemma}
				Let $ v$ be the solution of system \eqref{c2} with $v_{0} \in H^{-1}  (\mathbb{R}^{3})$. Then,
				\begin{equation} \label{c3}
					\|v_{0}\|^{2}_{H^{-1}} \leq C \Biggr(\int_{0}^{T} \|\varphi v(t)\|^{2}_{H^{-1}} \ dt + \|(1-\varphi(x/2))v_{0}\|^{2}_{H^{-2}} \Biggr).
				\end{equation}
			\end{lemma}
			\begin{proof}
				Again, let us argue by contradiction. If inequality \eqref{c3} is not verified, there exists a sequence $\{v_{n}\}$ of solutions to problem \eqref{c2} in $C([0,T];H^{-1}(\mathbb{R}^{3}))$ such that
				\begin{equation} \label{c4}
					1 = \|v_{n}(0)\|^{2}_{H^{-1}} \geq n \Biggr(\int_{0}^{T} \|\varphi v_{n}(t)\|^{2}_{H^{-1}} \ dt + \|(1-\varphi(x/2))v_{n}(0)\|^{2}_{H^{-2}} \Biggr).
				\end{equation}
				Up to a subsequence, we may assume that
				\begin{equation*}
					v_{n} \rightharpoonup v \ \mbox{  in  }\ L^{\infty}((0,T);H^{-1}(\mathbb{R}^{3})) \mbox{     weak*}
				\end{equation*}
				and
				\begin{equation} \label{c6}
					v_{n}(0) \rightharpoonup v(0) \mbox{   in   } H^{-1}(\mathbb{R}^{3}) \mbox{   weak,   }
				\end{equation}
				where $v \in C([0,T];H^{-1}(\mathbb{R}^{3}))$ is a solution of problem \eqref{c2}. By inequality \eqref{c4}, 
				$$
				\varphi v_{n} \rightarrow 0 \text{ (strongly) in } L^{2}((0,T);H^{-1}(\mathbb{R}^{3})).
				$$ 
				Since $$\varphi v_{n} \rightharpoonup \varphi v \text{ in $L^{\infty}((0,T);H^{-1}(\mathbb{R}^{3}))$ weak*,}$$ we conclude that 
				$\varphi v \equiv 0$. Here, the limit $v$ belongs only to $C([0,T];H^{-1}(\mathbb{R}^{3}))$, so the support statements below are understood in the sense of distributions. By continuity of $$t\mapsto\varphi v(t)\in H^{-1}(\mathbb{R}^{3}),$$ we have $\varphi v(t)=0$ for every $t\in[0,T]$; hence, for every $$\chi\in C_{c}^{\infty}(\{|x|>R+1\}),$$ we have $$\langle v(t),\chi\rangle=\langle\varphi v(t),\chi\rangle=0,$$ since $\varphi\equiv1$ on the support of $\chi$. In other words, $\supp v(t)\subset\overline{B_{R+1}(0)}$ for every $t\in[0,T]$. Applying Lemma \ref{UCP} with $s=-1$, $t_{1}=0$ and $t_{2}=T$, we get that $v \equiv 0$. In particular, $v(0)=0$. 
				
				Now, we claim that
				\begin{equation} \label{c7}
					\|\varphi(x/2)v_{n}(0)\|^{2}_{H^{-2}} \leq C \int_{0}^{T} \|\varphi v_{n}(t)\|^{2}_{H^{-1}} \ dt.
				\end{equation}
				To prove \eqref{c7}, introduce the function $\tilde{v}_{n}(x,t) = \varphi(x/2)v_{n}(x,t)$ which satisfies
				$$i\partial_{t}\tilde{v}_{n} + \Delta \tilde{v}_{n} = f_{n}$$ 
				%since
				%$$i\partial_{t}\tilde{v}_{n} = i\varphi(x/2)\partial_{t}v_{n};$$
				%$$\Delta \tilde{v}_{n}  = \Delta \Big( \varphi(x/2)v_{n} \Big) = [\Delta \varphi(x/2)]v_{n} + 2\nabla \varphi(x/2)\nabla v_{n} + \varphi(x/2) [\Delta v_{n}]$$
				%$$i\partial_{t}\tilde{v}_{n}  + \Delta \tilde{v}_{n}  = f_{n},$$
				where $f_{n} =  \Delta[\varphi(x/2)]\,v_{n} + 2\nabla[\varphi(x/2)]\cdot \nabla v_{n}$.
				Since $R\geq1$, we have $$\supp[\varphi(x/2)]\subset\{|x|\geq2R\}\subset\{|x|\geq R+1\}\subset\{\varphi=1\}.$$ By Duhamel's formula, $$\tilde v_{n}(0)=e^{-it\Delta}\tilde v_{n}(t)+i\int_{0}^{t}e^{-is\Delta}f_{n}(s)\,ds,$$ for every $t\in[0,T]$, and since $e^{it\Delta}$ is unitary on $H^{-2}(\mathbb{R}^{3})$, the Cauchy--Schwarz inequality gives
				\[
				\|\tilde v_{n}(0)\|^{2}_{H^{-2}}\leq 2\|\tilde v_{n}(t)\|^{2}_{H^{-2}}+2T\int_{0}^{T}\|f_{n}(s)\|^{2}_{H^{-2}}\,ds .
				\]
				Averaging in $t\in(0,T)$, we obtain the first inequality below. For the second one, note that $\tilde v_{n}=\varphi(x/2)\,(\varphi v_{n})$ and, since $\varphi\equiv1$ on the support of the coefficients $\Delta[\varphi(x/2)]$ and $\nabla[\varphi(x/2)]$ (all the derivatives of $\varphi$ vanish on $\{|x|\geq R+1\}$), $$f_{n}=\Delta[\varphi(x/2)]\,(\varphi v_{n})+2\nabla[\varphi(x/2)]\cdot\nabla(\varphi v_{n}).$$ Hence $$\|\tilde v_{n}(t)\|_{H^{-2}}+\|f_{n}(t)\|_{H^{-2}}\leq C\|\varphi v_{n}(t)\|_{H^{-1}},$$ and
				\begin{equation*}
					\begin{split}
						\|\tilde{v}_{n}(0)\|^{2}_{H^{-2}}  \leq & \frac{2}{T} \int_{0}^{T} \|\tilde{v}_{n}(t)\|^{2}_{H^{-2}} \ dt + 2T\int_{0}^{T} \|f_{n}(t)\|^{2}_{H^{-2}} \ dt  \\
						\leq & C\Big(\frac{1}{T}+T\Big) \int_{0}^{T}\|\varphi v_{n}(t)\|^{2}_{H^{-1}}\,dt,
					\end{split}
				\end{equation*}
				giving \eqref{c7}. Now, using \eqref{c4}, one has
				%$$\|v_{n}(0)\|^{2}_{H^{-2}} + \|2\varphi(x/2)v_{n}(0) - v_{n}(0)\|^{2}_{H^{-2}} = 2\Big(\|\varphi(x/2)v_{n}(0)\|^{2}_{H^{-2}} + \|(1-\varphi(x/2))v_{n}(0)\|^{2}_{H^{-2}}\Big)$$
				%(Parallelogram law)
				\begin{eqnarray*}
					\|v_{n}(0)\|^{2}_{H^{-2}} & \leq & 2\Big(\|\varphi(x/2)v_{n}(0)\|^{2}_{H^{-2}} + \|(1-\varphi(x/2))v_{n}(0)\|^{2}_{H^{-2}}\Big) \\
					& \leq & c \int_{0}^{T} \|\varphi v_{n}(t)\|^{2}_{H^{-1}} \ dt + 2\|(1-\varphi(x/2))v_{n}(0)\|^{2}_{H^{-2}} \rightarrow 0 ,
				\end{eqnarray*}
				that is, 
				\begin{equation} \label{c8}
					v_{n}(0) \rightarrow 0 \mbox{   strongly in   } H^{-2}(\mathbb{R}^{3}).
				\end{equation}
				
				Let $w_{n} =(1-\Delta)^{-1}v_{n}$. Then $w_{n} \in C([0,T];H^{1}(\mathbb{R}^{3}))$ is a solution of \eqref{c25} with $w_{n,0}=w_{n}(0)\in H^{1}(\mathbb{R}^{3})$. By the convergences \eqref{c6} and \eqref{c8}, we can ensure that
				\begin{equation*}
					w_{n}(0) \rightharpoonup 0 \mbox{  in  } H^{1}(\mathbb{R}^{3}) \mbox{  weakly  }
				\end{equation*} 
				and
				\begin{equation}\label{c10}
					w_{n} \rightarrow 0 \mbox{  in  } C([0,T];L^{2}(\mathbb{R}^{3})) \mbox{  strongly.  }
				\end{equation}
				Now, split $\varphi w_{n}$ as
				$$\varphi w_{n} = (1-\Delta)^{-1}(\varphi v_{n}) - (1-\Delta)^{-1}[\varphi,(1-\Delta)]w_{n}.$$
				Observe that  the operator $ [\varphi,(1-\Delta)]$ maps $L^{2}(\mathbb{R}^{3})$ continuously into $H^{-1}(\mathbb{R}^{3})$. So, due to the convergence \eqref{c10}, we get that 
				\begin{equation}\label{c11}
					(1-\Delta)^{-1}[\varphi,(1-\Delta)]w_{n} \rightarrow 0 \mbox{  in  } C([0,T];H^{1}(\mathbb{R}^{3})).	
				\end{equation}
				On the other hand, by \eqref{c4},
				\begin{equation}\label{c12}
					(1-\Delta)^{-1}(\varphi v_{n})\rightarrow 0 \mbox{  in  } L^{2}((0,T);H^{1}(\mathbb{R}^{3})).
				\end{equation}
				Therefore, by the convergences \eqref{c11} and \eqref{c12} above, it follows that
				\begin{equation*}
					\varphi w_{n} \rightarrow 0 \mbox{  in  } L^{2}((0,T);H^{1}(\mathbb{R}^{3})).
				\end{equation*}
				Since $w_{n}$ satisfies \eqref{c25}, using Lemma \ref{obs1}, more precisely, the observability inequality \eqref{c26}, we conclude that 
				$$w_{n}(0) \rightarrow 0 \mbox{  in  } H^{1}(\mathbb{R}^{3}) \mbox{  strongly,  }$$
				and so
				$$v_{n}(0) \rightarrow 0 \mbox{  in  } H^{-1}(\mathbb{R}^{3}) \mbox{  strongly,}$$
				which is a contradiction to the fact that $ \|v_{n}(0)\|^{2}_{H^{-1}}=1$, for all $n$. This finishes the proof.
			\end{proof}
			
			\noindent\textbf{Step 3}: \textit{Proof of the observability inequality \eqref{c1}.}
			
			\vspace{0.1cm}
			If \eqref{c1} is false, then there exists a sequence $\{v_{n}\}$ of solutions to
			\eqref{c2} in $C([0,T];H^{-1}(\mathbb{R}^{3}))$ such that
			\begin{equation} \label{c13}
				1 = \|v_{n}(0)\|^{2}_{H^{-1}} \geq n \int_{0}^{T} \|\varphi v_{n}(t)\|^{2}_{H^{-1}} \ dt, \ \forall n\geq 1.
			\end{equation}
			Extracting a subsequence, still denoted by the same indexes, we have that
			\begin{equation*} 
				v_{n} \rightharpoonup v \ \mbox{  in  }\ L^{\infty}((0,T);H^{-1}(\mathbb{R}^{3})) \mbox{     weak*}
			\end{equation*}
			and
			\begin{equation*} 
				v_{n}(0) \rightharpoonup v(0) \mbox{   in   } H^{-1}(\mathbb{R}^{3}) \mbox{   weak, }
			\end{equation*}
			for some solution $v \in C([0,T];H^{-1}(\mathbb{R}^{3}))$ of the system \eqref{c2}. Note that 
			$$
			\varphi v_{n} \rightharpoonup \varphi v \text{ in $L^{\infty}(0,T;H^{-1}(\mathbb{R}^{3}))$ weak*}
			$$ and this, combined with \eqref{c13} ($\varphi v_{n} \rightarrow 0$ in $L^{2}\big((0,T);H^{-1}(\mathbb{R}^{3})\big)$), yields $\varphi v \equiv 0$ and, hence, as in Step 2, $\supp v(t)\subset\overline{B_{R+1}(0)}$ in the sense of distributions, for every $t\in[0,T]$. By Lemma \ref{UCP} (with $s=-1$), we deduce that $v \equiv 0$ on $(0,T)\times\mathbb{R}^{3}$. 
			On the other hand, the sequence ${(1 - \varphi(x/2))v_{n}(0)}$ is bounded in $H^{-1}(\mathbb{R}^{3})$ and has
			support contained in the closed ball $\overline{B_{2R+2}(0)}$. Let $\chi\in C_{c}^{\infty}(\mathbb{R}^{3})$ be such that $\chi\equiv1$ in a neighborhood of $\overline{B_{2R+2}(0)}$, so that $(1-\varphi(x/2))v_{n}(0)=\chi\,(1-\varphi(x/2))v_{n}(0)$. The multiplication by $\chi$ is compact from $H^{-1}(\mathbb{R}^{3})$ into $H^{-2}(\mathbb{R}^{3})$, by duality, since it is compact from $H^{2}(\mathbb{R}^{3})$ into $H^{1}(\mathbb{R}^{3})$ by Rellich's theorem. Therefore, extracting a subsequence, we may assume that it converges strongly in $H^{-2}(\mathbb{R}^{3})$. Moreover, its limit is necessarily $0$ since $${(1 - \varphi(x/2))v_{n}(0)}\rightharpoonup 0 \text{ in $H^{-2}(\mathbb{R}^{3})$}.$$ Using inequality \eqref{c3}, we conclude that $\|v_{n}(0)\|_{H^{-1}} \rightarrow 0$, which contradicts \eqref{c13}. This proves the 
			desired observability inequality \eqref{c1} and finishes the proof of Proposition \ref{obs2}.
		\end{proof}
		
		We are now in a position to prove Theorem \ref{lin}.
		
		\begin{proof}[Proof of Theorem \ref{lin}] We prove Theorem \ref{lin} using Hilbert's uniqueness method. Since the equation \eqref{linear} is linear and well-posed backward in time, it suffices to show that, for every $u_{0}\in H^{1}(\mathbb{R}^{3})$, there exists a control $h$ such that the solution of the backward problem \eqref{kkkkk} below, with final datum $u(T)=0$, satisfies $u(0)=u_{0}$; by uniqueness, the solution of \eqref{linear} with this control then satisfies $u(T)=0$.
			
			Now, consider the two systems
			\begin{equation}\label{kkkkk}
				\left\{
				\begin{array}{lr}
					i\partial_{t} u + \Delta u = \varphi(x) h(x,t) \ \mbox{on} \ [0,T] \times \mathbb{R}^{3} , \\
					u(T) = 0 , 
				\end{array}
				\right.
			\end{equation}
			with $\varphi(x)$ satisfying \eqref{varphi},  and 
			\begin{equation*}
				\left\{
				\begin{array}{lr}
					i\partial_{t} v + \Delta v = 0 \ \mbox{on} \ [0,T] \times \mathbb{R}^{3} , \\
					v(0) = v_{0} \in H^{-1}(\mathbb{R}^{3}).
				\end{array}
				\right.
			\end{equation*}
			Multiplying the equation of the first system by $\overline{v}$ and integrating by parts, we obtain
			$$ i\int_{\mathbb{R}^{3}} \Big[\overline{v}(T)u(T) -\overline{v_{0}}u(0)\Big] \ dx = \int_{0}^{T} \int_{\mathbb{R}^{3}}\varphi(x) h(x,t) \overline{v(x,t)}\ dx dt.$$
			This identity is first established for smooth data, say $v_{0}\in H^{2}(\mathbb{R}^{3})$ and $h\in C([0,T];H^{2}(\mathbb{R}^{3}))$, for which all the integrals are classical and the integrations by parts are justified. It then extends by density to $v_{0}\in H^{-1}(\mathbb{R}^{3})$ and $h\in C([0,T];H^{1}(\mathbb{R}^{3}))$, using the continuity of $v_{0}\mapsto v$ in $C([0,T];H^{-1}(\mathbb{R}^{3}))$, of $h\mapsto u$ in $C([0,T];H^{1}(\mathbb{R}^{3}))$, and interpreting all the integrals as $H^{-1}$--$H^{1}$ duality pairings.
			Hence, taking $L^2(\mathbb{R}^3)$ as pivot space, one has
			\begin{equation}\label{c16a}
				\langle v_{0},-iu(0)\rangle=  \int_{0}^{T} \langle \varphi(x)v,h(t) \rangle \ dt,
			\end{equation}
			where $\langle \cdot,\cdot \rangle$ denotes the duality pairing between $H^{-1}\left(\mathbb{R}^3\right)$ and $H^1\left(\mathbb{R}^3\right)$, which is antilinear in the first argument and extends $\langle f,g\rangle=\int_{\mathbb{R}^{3}}\overline{f}\,g\,dx$ (all the spaces are complex).  
			
			Now, consider the (Riesz) isomorphism between the complex spaces $H^{1}$ and $H^{-1}$ defined as follows:
			\begin{equation*}
				\Lambda : H^{1}(\mathbb{R}^{3}) \rightarrow H^{-1}(\mathbb{R}^{3})
			\end{equation*}
			 given by 
			 \begin{equation}\label{defiso}
			 	\langle \Lambda\phi,v\rangle = \int_{\mathbb{R}^{3}} \overline{\phi(x)}\,v(x) \ dx + \int_{\mathbb{R}^{3}} \overline{\nabla \phi(x)} \cdot \nabla v(x) \ dx, \ \ \phi,v \in H^{1}(\mathbb{R}^{3}),
			 \end{equation}
For any $v_0 \in H^{-1}\left(\mathbb{R}^3\right)$, let $h(t)=\Lambda^{-1}(\varphi v(t))\left(h \in C\left([0, T] ; H^1\left(\mathbb{R}^3\right)\right)\right)$ and let $u$ be the corresponding solution of system \eqref{kkkkk}. Finally, set 
			\begin{equation} \label{gamma}
					\Gamma : H^{-1}(\mathbb{R}^{3}) \rightarrow H^{1}(\mathbb{R}^{3})
			\end{equation}
			given by $\Gamma\left(v_0\right)=-i u(\cdot, 0)$. Then, we have
			$$
			\left\langle v_0, \Gamma\left(v_0\right)\right\rangle=\int_0^T\|\varphi v(t)\|_{H^{-1}\left(\mathbb{R}^3\right)}^2 d t \geqslant c\left\|v_0\right\|_{H^{-1}\left(\mathbb{R}^3\right)}^2 ,
			$$
			in view of the observability inequality \eqref{c1} and identity \eqref{c16a}, where we used that $\Lambda=I-\Delta$ is an isometry from $H^{1}(\mathbb{R}^{3})$ onto $H^{-1}(\mathbb{R}^{3})$, so that $\langle f,\Lambda^{-1}f\rangle=\|f\|^{2}_{H^{-1}}$ for $f\in H^{-1}(\mathbb{R}^{3})$. Thus the sesquilinear form $a(v_{0},w_{0}):=\langle w_{0},\Gamma v_{0}\rangle$ on $H^{-1}(\mathbb{R}^{3})$ is continuous and satisfies $\mathrm{Re}\,a(v_{0},v_{0})\geq c\|v_{0}\|^{2}_{H^{-1}}$. Since the operator $\Gamma$ is continuous and coercive in this sense, it follows from the Lax-Milgram theorem that $\Gamma$ defines an isomorphism; given $u_{0}\in H^{1}(\mathbb{R}^{3})$, the choice $v_{0}=\Gamma^{-1}(-iu_{0})$ provides a control $$h=\Lambda^{-1}(\varphi v)\in C([0,T];H^{1}(\mathbb{R}^{3}))$$ driving $u_{0}$ to $0$ at time $T$. Since $\varphi=0$ in $B_{R}(0)$, the control force $\varphi h$ is supported in $[0,T]\times\{|x|\geq R\}$, and this concludes the proof of Theorem \ref{lin}.
		\end{proof}
		
		\subsection{Nonlinear system: Proof of Theorem \ref{main}}\label{sub32} The proof is based on a perturbation argument due to E. Zuazua \cite{Zuazua1}. Throughout this subsection, $T>0$ is fixed, $I=[0,T]$, $N(u)=\pm|u|^{4}u$, and $\|\cdot\|_{X(I)}$, $C_{0}$, $\eta_{0}$ are as in Theorem \ref{t2.4}. To use it, consider the free adjoint system with initial datum in $H^{-1}$
and the forced nonlinear system with zero final datum, namely
		\begin{equation*}
			\left\{
			\begin{array}{lr}
				i\partial_{t} \Phi + \Delta \Phi = 0 \mbox{ on} \ [0,T] \times \mathbb{R}^{3} , \\
				\Phi(0) = \Phi_{0} \in H^{-1}(\mathbb{R}^{3}) ,
			\end{array}
			\right.
		\end{equation*}
		and
		\begin{equation}\label{c18}
			\left\{
			\begin{array}{lr}
				i\partial_{t} u + \Delta u +N(u)= A\Phi \mbox{ on} \ [0,T] \times \mathbb{R}^{3} , \\
				u(T) = 0,
			\end{array}
			\right.
		\end{equation}
		respectively, where $A$ is defined by $$A\Phi: = \varphi(x)\Lambda^{-1}(\varphi(x) \Phi),$$
		where $\Lambda$ is given in \eqref{defiso}. Note that $A\Phi=\varphi h$ with $$h=\Lambda^{-1}(\varphi\Phi)\in C([0,T];H^{1}(\mathbb{R}^{3})).$$ The function $h$ is not compactly supported in general, but the forcing term $A\Phi$ is supported in $[0,T]\times\{|x|\geq R\}$.

		\vspace{0.2cm}
		\noindent\textbf{Step 1:} \textit{System \eqref{c18} is well posed on the whole interval $[0,T]$.}
		\vspace{0.1cm}

		Since $e^{it\Delta}$ is unitary on $H^{-1}(\mathbb{R}^{3})$, we have $\|\Phi(t)\|_{H^{-1}}=\|\Phi_{0}\|_{H^{-1}}$ for every $t$. The multiplication by $\varphi$ is bounded on $H^{-1}(\mathbb{R}^{3})$ and on $H^{1}(\mathbb{R}^{3})$, with norms bounded by $c_{\varphi}:=1+\|\nabla\varphi\|_{L^{\infty}}$, and $$\Lambda^{-1}:H^{-1}(\mathbb{R}^{3})\to H^{1}(\mathbb{R}^{3})$$ is an isometry. Hence
		\begin{equation}\label{A-bound}
			\|A\Phi(t)\|_{H^{1}}\leq M\|\Phi_{0}\|_{H^{-1}},
		\end{equation}
for  $t\in I$ and $M:=c_{\varphi}^{2}$, thus $$\|A\Phi\|_{L^{1}(I;H^{1})}\leq MT\|\Phi_{0}\|_{H^{-1}}.$$
Recall that $\eta_{0}>0$ is the constant given by Theorem \ref{t2.4}. Note that Theorem \ref{t2.4} remains valid, with the same constant $C_{0}$, if $\eta_{0}$ is replaced by any smaller positive number. We shall use this fact in Step 4 below, where $\eta_{0}$ is required to satisfy the additional constraint \eqref{condeta}, which depends only on $T$, $\varphi$, $C_{S}$ and $C_{Sob}$. Set
		$$K_{0}:=\frac{\eta_{0}}{MT},$$
		so that $\|\Phi_{0}\|_{H^{-1}}\leq K_{0}$ implies $\|A\Phi\|_{L^{1}(I;H^{1})}\leq\eta_{0}$, which is precisely the smallness condition of Theorem \ref{t2.4}. If $\|\Phi_{0}\|_{H^{-1}}\leq K_{0}$, Theorem \ref{t2.4}, applied with $t_{0}=T$, $u_{T}=0$ and $g=A\Phi$, gives a unique solution $u\in C(I;H^{1}(\mathbb{R}^{3}))\cap L^{10}(I;W^{1,30/13}(\mathbb{R}^{3}))$ of \eqref{c18}, defined on the whole interval $I$, such that
		\begin{equation}\label{c22}
			\|u\|_{X(I)}\leq 2C_{0}MT\,\|\Phi_{0}\|_{H^{-1}}.
		\end{equation}
		Moreover, if $u_{1},u_{2}$ are the solutions associated with $\Phi^{1}_{0},\Phi^{2}_{0}$, with $\|\Phi^{j}_{0}\|_{H^{-1}}\leq K_{0}$, then \eqref{lip} and \eqref{A-bound} give
		\begin{equation}\label{difference-u-Z}
			\|u_{1}-u_{2}\|_{X(I)}\leq 2C_{0}MT\,\|\Phi^{1}_{0}-\Phi^{2}_{0}\|_{H^{-1}}.
		\end{equation}
		Consequently, the operator
		\begin{equation*}
			\begin{aligned}
				\mathcal{L}: \{\Phi_{0}\in H^{-1}(\mathbb{R}^{3}):\ \|\Phi_{0}\|_{H^{-1}}\leq K_{0}\} & \rightarrow H^{1}(\mathbb{R}^{3}) \\
				\Phi_0 & \mapsto \mathcal{L} \Phi_0=u(0)
			\end{aligned}
		\end{equation*}
		is well defined. The goal is to show that the equation $\mathcal{L}\Phi_{0}=u_{0}$ can be solved for every $u_{0}$ in a small neighborhood of the origin of $H^{1}(\mathbb{R}^{3})$.

		\vspace{0.2cm}
		\noindent\textbf{Step 2:} \textit{Reduction to a fixed-point problem.}
		\vspace{0.1cm}

		Split $u$ as $u = v + \Psi$, where $\Psi$ is the solution of
		\begin{equation*}
			\left\{
			\begin{array}{lr}
				i\partial_{t} \Psi + \Delta \Psi = A\Phi \ \mbox{on} \ [0,T] \times \mathbb{R}^{3},\\
				\Psi(T) = 0 ,
			\end{array}
			\right.
		\end{equation*}
		and $v$ is the solution of
		\begin{equation}\label{c20}
			\left\{
			\begin{array}{lr}
				i\partial_{t} v + \Delta v =-N(u)\ \mbox{on} \ [0,T] \times \mathbb{R}^{3},\\
				v(T) = 0.
			\end{array}
			\right.
		\end{equation}
		Both are linear problems with forcing terms $A\Phi\in L^{1}(I;H^{1})$ and $N(u)$, $\nabla N(u)\in L^{2}(I;L^{6/5})$ (see \eqref{N-est} below), so that, by Lemma \ref{strichartz} and the density argument used in the proof of Theorem \ref{t2.4}, $v,\Psi\in C(I;H^{1}(\mathbb{R}^{3}))$. By the definition \eqref{gamma} of $\Gamma$ (with $h=\Lambda^{-1}(\varphi\Phi)$), we have $\Gamma\Phi_{0}=-i\Psi(0)$, that is, $\Psi(0)=i\Gamma\Phi_{0}$. Therefore
		$$\mathcal{L} \Phi_{0} = \mathcal{J} \Phi_{0} + i\Gamma\Phi_{0}, \qquad \text{with}\quad \mathcal{J} \Phi_{0} := v(0),$$
		and, since $\Gamma:H^{-1}\to H^{1}$ is an isomorphism (Theorem \ref{lin}), $\mathcal{L}\Phi_{0} = u_{0}$ if and only if $\Phi_{0}$ is a fixed point of
		$$
		\mathcal{B}\Phi_0:=(i\Gamma)^{-1}\big(u_0- \mathcal{J} \Phi_0\big)=-i\,\Gamma^{-1}\big(u_0- \mathcal{J} \Phi_0\big).
		$$
		If $\Phi_{0}$ is such a fixed point, then $u$ solves \eqref{eq1} with $$h=\Lambda^{-1}(\varphi\Phi)\in C([0,T];H^{1}(\mathbb{R}^{3})),$$ $u(0)=u_{0}$ and $u(T)=0$, which is the conclusion of Theorem \ref{main}. Let $\gamma:=\|\Gamma^{-1}\|_{\mathcal{L}(H^{1},H^{-1})}$, which depends only on $T$ and $\varphi$. Then
		\begin{equation}\label{bb}
		\begin{cases}
			\|\mathcal{B}\Phi_{0}\|_{H^{-1}}  \leq \gamma\big(\|u_{0}\|_{H^{1}}+\|v(0)\|_{H^{1}}\big), \\ \vspace{0.2cm}
			\|\mathcal{B}\Phi_{0}^{1} - \mathcal{B}\Phi_{0}^{2}\|_{H^{-1}}  \leq  \gamma \|v_{1}(0)-v_{2}(0)\|_{H^{1}},
			\end{cases}
		\end{equation}
		where $v_{j}$ is the solution of \eqref{c20} associated with $u_{j}$, $j=1,2$.

		\vspace{0.2cm}
		\noindent\textbf{Step 3:} \textit{Estimates for $\mathcal{J}$.}
		\vspace{0.1cm}

		By H\"older's inequality ($\frac{4}{10}+\frac{1}{10}=\frac12$ in time, $\frac{4}{10}+\frac{13}{30}=\frac56$ in space) and Lemma \ref{sobolev},
		\begin{equation}\label{N-est}
		\begin{cases}
			\|N(u)\|_{L^{2}(I;L^{6/5})}\leq\|u\|^{4}_{S(I)}\|u\|_{Z(I)},\\ \vspace{0.2cm}
			\|\nabla N(u)\|_{L^{2}(I;L^{6/5})}\leq5\|u\|^{4}_{S(I)}\|\nabla u\|_{Z(I)},
			\end{cases}
		\end{equation}
		so that $$\|N(u)\|_{L^{2}(I;L^{6/5})}+\|\nabla N(u)\|_{L^{2}(I;L^{6/5})}\leq 6C_{Sob}^{4}\|u\|^{5}_{X(I)}.$$ Moreover, using the pointwise bounds
		$$|N(u_{1})-N(u_{2})|\leq 5(|u_{1}|^{4}+|u_{2}|^{4})|u_{1}-u_{2}|$$ and
		$$|\nabla N(u_{1})-\nabla N(u_{2})|\leq C\big(|u_{1}|^{4}+|u_{2}|^{4}\big)|\nabla(u_{1}-u_{2})|+C\big(|u_{1}|^{3}+|u_{2}|^{3}\big)\big(|\nabla u_{1}|+|\nabla u_{2}|\big)|u_{1}-u_{2}|,$$
		H\"older's inequality and Lemma \ref{sobolev} (applied to $u_{1}$, $u_{2}$ and $u_{1}-u_{2}$) yield
		\begin{equation}\label{N-diff}
			\begin{cases}
				\|N(u_{1})-N(u_{2})\|_{L^{2}(I;L^{6/5})} \leq 5\big(\|u_{1}\|^{4}_{S(I)}+\|u_{2}\|^{4}_{S(I)}\big)\|u_{1}-u_{2}\|_{Z(I)},\\ \vspace{0.2cm}
				\|\nabla N(u_{1})-\nabla N(u_{2})\|_{L^{2}(I;L^{6/5})} \leq C\big(\|u_{1}\|^{4}_{X(I)}+\|u_{2}\|^{4}_{X(I)}\big)\|u_{1}-u_{2}\|_{X(I)}.
			\end{cases}
		\end{equation}
		We stress that the undifferentiated difference is estimated with $\|u_{1}-u_{2}\|_{Z(I)}$ and not with $\|u_{1}-u_{2}\|_{S(I)}$, since the space exponents must satisfy $\frac{4}{10}+\frac{13}{30}=\frac56$.

		Duhamel's formula for \eqref{c20} reads $$v(0)=-i\int_{0}^{T}e^{-i\tau\Delta}N(u(\tau))\,d\tau.$$ Applying \eqref{item_ii} with $t_{0}=T$, the pair $(\infty,2)$ on the left-hand side and the endpoint pair $(2,6)$ on the right-hand side (see Remark \ref{pairs}), and using \eqref{N-est} and \eqref{N-diff}, we obtain the following:

\vspace{0.2cm}
		\noindent There exist constants $C_{3},C_{4}>0$, depending only on $C_{S}$ and $C_{Sob}$, such that
		\begin{equation}\label{c21}
		\begin{cases}
			\|v(0)\|_{H^{1}} \leq C_{3}\|u\|^{5}_{X(I)}, \\
			\|v_{1}(0)-v_{2}(0)\|_{H^{1}} \leq C_{4}\big(\|u_{1}\|^{4}_{X(I)}+\|u_{2}\|^{4}_{X(I)}\big)\|u_{1}-u_{2}\|_{X(I)}.
			\end{cases}
		\end{equation}

		\vspace{0.2cm}
		\noindent\textbf{Step 4:} \textit{$\mathcal{B}$ is a contraction on a small ball.}
		\vspace{0.1cm}

		Let $0<K\leq K_{0}$ and $$B_{K}=\{\Phi_{0}\in H^{-1}(\mathbb{R}^{3}):\|\Phi_{0}\|_{H^{-1}}\leq K\}.$$ For $\Phi_{0},\Phi^{1}_{0},\Phi^{2}_{0}\in B_{K}$, estimates \eqref{c22}, \eqref{difference-u-Z}, \eqref{bb} and \eqref{c21} give
$$
			\|\mathcal{B}\Phi_{0}\|_{H^{-1}} \leq \gamma\|u_{0}\|_{H^{1}}+\gamma C_{3}(2C_{0}MT)^{5}K^{5},$$
			and
			$$
			\|\mathcal{B}\Phi_{0}^{1} - \mathcal{B}\Phi_{0}^{2}\|_{H^{-1}} \leq 2\gamma C_{4}(2C_{0}MT)^{5}K^{4}\,\|\Phi^{1}_{0}-\Phi^{2}_{0}\|_{H^{-1}}.
			$$
These two estimates hold for every $K\in(0,K_{0}]$. In order to conclude that $\mathcal{B}$ maps $B_{K}$ into itself and is a contraction, we need $K$ to satisfy, in addition,
		\begin{equation}\label{condK}
			\gamma C_{3}(2C_{0}MT)^{5}K^{4}\leq\frac12 \qquad\text{and}\qquad 2\gamma C_{4}(2C_{0}MT)^{5}K^{4}\leq\frac12,
		\end{equation}
		that is, $K\leq K_{1}$, where
		\begin{equation}\label{choiceK}
			K_{1}:=\Big(2\gamma\max\{C_{3},2C_{4}\}(2C_{0}MT)^{5}\Big)^{-1/4}.
		\end{equation}
		To guarantee that the choice of $K\leq K_{0}$ is compatible with \eqref{condK}, we impose the following constraint on the constant $\eta_{0}$ of Theorem \ref{t2.4}:
		\begin{equation}\label{condeta}
			\eta_{0}\leq MT K_{1},\qquad\text{that is,}\qquad K_{0}=\frac{\eta_{0}}{MT}\leq K_{1}.
		\end{equation}
		As observed in Step 1, this is not a restriction, since $\eta_0$ may be reduced in Theorem 2.3. Note that $K_1$ does not depend on $\eta_0$, since the constants $C_0$, $C_3$, $C_4$, $M$ and $\gamma$ do not. Under the constraint \eqref{condeta}, every $K\in(0,K_{0}]$ satisfies \eqref{condK}, and we choose
		$$K:=K_{0}=\frac{\eta_{0}}{MT}.$$

Finally, set $\delta:=K/(2\gamma)$ and let $\|u_{0}\|_{H^{1}}\leq\delta$. For $\Phi_{0},\Phi^{1}_{0},\Phi^{2}_{0}\in B_{K}$, the estimates above and \eqref{condK} give
		$$\|\mathcal{B}\Phi_{0}\|_{H^{-1}}\leq\gamma\delta+\Big(\gamma C_{3}(2C_{0}MT)^{5}K^{4}\Big)K\leq\frac{K}{2}+\frac{K}{2}=K$$
		and
		$$\|\mathcal{B}\Phi_{0}^{1} - \mathcal{B}\Phi_{0}^{2}\|_{H^{-1}} \leq \frac12\,\|\Phi^{1}_{0}-\Phi^{2}_{0}\|_{H^{-1}}.$$
		Therefore, $\mathcal{B}$ maps the closed ball $B_{K}$ of $H^{-1}(\mathbb{R}^{3})$ into itself and is a contraction on $B_{K}$, with constant $\frac12$. By the Banach fixed-point theorem, $\mathcal{B}$ has a unique fixed point $\Phi_{0}\in B_{K}$, and the corresponding solution $u$ of \eqref{c18} satisfies $u(0)=u_{0}$ and $u(T)=0$. Note that $\delta$ depends only on $T$ and on $\varphi$ (through $M$ and $\gamma$), and on the constants $C_{S}$, $C_{Sob}$ (through $\eta_{0}$, $C_{0}$, $C_{3}$ and $C_{4}$). This completes the proof of Theorem~\ref{main}.
		\qed

		%%%%%%%%%%%%%%%%%%%%%%%%%%%% sec 4 %%%%%%%%%%%%%%%%%%%%%%%%%%%
		
		\section{Comments and open issues}\label{Sec4}
		In this work, we have studied the local null controllability of the energy-critical C-NLS in dimension $3$. Considering the %defocusing%
		 critical nonlinear Schr\"{o}dinger equation \eqref{eq1}, we showed this system to be controllable in ${H}^{1}(\mathbb{R}^{3})$--level, that is, it satisfies $u(T) =0$ for  $u_{0}\in{H}^{1}(\mathbb{R}^{3})$ satisfying $\|u_{0}\|_{H^{1}} \leq \delta$,  and $h(x,t) \in C([0,T];H^{1}(\mathbb{R}^{3}))$, the control force $\varphi h$ being supported in $[0,T]\times\{|x|\geq R\}$. Concerning our main results, Theorems \ref{main} and \ref{lin}, the following remarks are worth mentioning.

		\begin{itemize}
			\item The main tools to achieve these results were:
			
			-- \textit{Strichartz estimates} for the Schr\"odinger operator \cite{cazenave_book,KeelTao}, which give the well-posedness theory for system \eqref{eq1}; 
			
			-- \textit{A unique continuation property} for the linear system associated with \eqref{eq1}, which we prove in Lemma \ref{UCP} by means of the Paley--Wiener--Schwartz theorem; 
			
			-- \textit{A perturbation argument}, as presented in \cite{Zuazua1}, which allows one to extend the result for the nonlinear system \eqref{eq1}. 
			
			\vspace{0.1cm}
			
			\item For the $3d$--case, Laurent \cite{Laurent1} proved global internal controllability in large time for the defocusing cubic nonlinear Schr\"odinger equation (that is, with the nonlinearity $|u|^{2}u$ instead of $|u|^{4}u$) on some compact manifolds of dimension $3$, by combining a stabilization result with the local controllability near $0$. Our result can be seen as a counterpart of this local step for the energy-critical nonlinearity $|u|^{4}u$, in the non-compact setting of $\mathbb{R}^3$. 
			
			\vspace{0.1cm} 
					\item For small initial data, the sign of the nonlinearity is not relevant, thus, our proof applies equally to both nonlinearities $|u|^{4}u$ and $- |u|^{4}u$. 
						\vspace{0.1cm} 
			
			\item Since we are working on the whole space $\mathbb{R}^3$ and the control takes the form $\varphi(x)h(x,t)$ in \eqref{eq1}, every ray $\{x_{0}+s\xi:\ s\geq0\}$ with $|\xi|=1$ enters the open set $\omega=\{|x|>R+1\}$, where $\varphi=1$, before the time $T_{0}$, for any $T_{0}>2(R+1)$. Hence the geometric control condition (see, for instance, \cite{BaLeRau,RaTa}), formulated as usual for the Schr\"odinger equation with unit-speed rays and a geometric time $T_{0}$ which is independent of the control time $T$ (see, e.g., \cite[Assumption 1]{Laurent1}), is automatically satisfied. The low-frequency regime, for which the trajectories $x_{0}+2t\xi$ may stay in $B_{R}(0)$ for times of order $R/|\xi|$, is handled by the compactness--uniqueness argument of Proposition \ref{obs2}, which gives observability for every $T>0$.  On the other hand, Ta\"ufer \cite{Taufer} proved that, for the \emph{linear} Schr\"odinger equation in $\mathbb{R}^{d}$, observability holds for nonempty periodic open sets, which in general do not satisfy the GCC. Whether Theorem \ref{main} remains true for such control regions is an open question; see item $(\mathcal{C})$ below.
		\end{itemize}

		Thus, our work gives the first step to understanding control problems for the C-NLS in dimension $3$, and, therefore, some important open issues appear. Let us present them below.

		\begin{itemize}
			
			\item[($\mathcal{A})$]  \textbf{Stabilization problem:} Can one find a feedback control law $f=\varphi Ku$, acting in the exterior control region, such that the resulting closed-loop system in the defocusing case $$i\partial_{t} u + \Delta u  - |u|^{4}u =\varphi(x)Ku,\quad x\in\mathbb{R}^3,\ t\geq0,$$ is globally well posed for initial data in $H^1(\mathbb{R}^3)$ and has the origin as a globally asymptotically stable equilibrium in $H^1(\mathbb{R}^3)$? The feedback should generate admissible controls $h=Ku\in C([0,\tau];H^1(\mathbb{R}^3))$ on every finite interval $[0,\tau]$.

			\vspace{0.1cm} 
			
			\item[($\mathcal{B}$)] \textbf{Global control problem:} If the answer to the previous question is positive, another natural issue is to obtain a global control result for the system \eqref{eq1}, i.e., a control result for large data. If such a localized feedback yields global solutions converging to zero in $H^1(\mathbb{R}^3)$, one can apply it until the solution enters the small ball of Theorem \ref{main}, and then apply the local null control on a fixed additional interval. This gives null controllability for arbitrary $H^1$ initial data in a time that may depend on the initial datum; it does not establish controllability in an arbitrary prescribed time. If continuity of $h$ across the switching time is required, the feedback can be smoothly switched off over a sufficiently short transition interval, chosen so that the state remains in the small ball, and the local control can be chosen continuous and vanishing at the endpoints of the control interval. Indeed, let $\theta\in C^{\infty}_{c}((0,T))$ be real-valued, with $|\theta|\geq1$ on a subinterval $J\subset(0,T)$, and replace the forcing term $A\Phi$ by $\theta(t)^{2}A\Phi$, that is, $h(t)=\theta(t)^{2}\Lambda^{-1}(\varphi\Phi(t))$. The corresponding operator $\Gamma_{\theta}$ satisfies $$\langle v_{0},\Gamma_{\theta}v_{0}\rangle=\int_{0}^{T}\theta^{2}\|\varphi v(t)\|^{2}_{H^{-1}}\,dt\geq\int_{J}\|\varphi v(t)\|^{2}_{H^{-1}}\,dt\geq c\|v_{0}\|^{2}_{H^{-1}},$$ by Proposition \ref{obs2} applied on $J$ (by time translation and conservation of the $H^{-1}$ norm). Moreover, $$\|\theta^{2}A\Phi\|_{L^{1}(0,T;H^{1})}\leq\|\theta\|^{2}_{L^{\infty}}MT\|\Phi_{0}\|_{H^{-1}},$$ so that the proof of Subsection \ref{sub32} applies, with constants depending additionally on \(\theta\). In the defocusing case, the global well-posedness and scattering result of Colliander, Keel, Staffilani, Takaoka and Tao \cite{CKSTT} for arbitrary data in $\dot H^{1}(\mathbb{R}^{3})$ is the natural starting point for such a large-data analysis.
			
			\vspace{0.1cm} 
			
			\item[($\mathcal{C})$] \textbf{Geometric control condition (GCC) or thick-set conditions:} Our work considered a control force $\varphi h$ supported in the exterior region $\{|x|\geq R\}$. From the perspective of the observation region of the linear Schr\"odinger equation, an interesting open problem appears: Does the local controllability of the critical equation \eqref{eq1} hold for more general control regions, for instance for the periodic sets considered in \cite{Taufer}, which do not necessarily satisfy the GCC?
			
		\end{itemize}
		\section{Declarations}
		\subsection{Competing interests}
			There are no competing interests.
		\subsection{Funding} The Article Processing Charge (APC) for the publication of this research was funded by the Coordenação
de Aperfeiçoamento de Pessoal de Nível Superior - Brasil (CAPES) (ROR identifier: 00x0ma614).  
		
			\medskip
			
			P. Braz e Silva was partially supported by 
			\begin{itemize}
			\item CAPES/MATH-AMSUD  \#8881.520205/2020-01,
			\item CAPES/PRINT \#88887.311962/2018-00,
			\item CAPES/COFECUB \#88887.879175/2023-00,
			\item  CNPq \#421573/2023-6, \#305233/2021-1, \#308758/2018-8, \#432387/2018-8.
			\end{itemize} 
			
					\medskip
					
			R. Capistrano--Filho was partially supported by 
			\begin{itemize} 
			\item CAPES/MATH-AMSUD \#8881.520205/2020-01, 
			\item CAPES/PRINT \#88887.311962/2018-00, 
			\item CAPES/COFECUB \#88887.879175/2023-00, 
			\item CNPq \# 301744/2025-4, \#421573/2023-6 
			\item PROPG (UFPE). 
			\end{itemize}
			
					\medskip
					
		J. Carvalho was partially supported by
		\begin{itemize} 
		\item CNPq
		\item CAPES-MATHAMSUD \#88887.700172/2022-00 
		\item FACEPE BFD-0014-1.01/23.
		\end{itemize} 
		
		\subsection{Authors' contributions}	
			 R. Capistrano-Filho has suggested the work.  J. Carvalho, R. Capistrano-Filho, and D. Dos Santos Ferreira have participated in all phases of the project.  P. Braz e Silva has participated in the first discussions of the results and seminars in Brazil and in the final review. 
			 \subsection{Acknowledgment}
The authors would like to thank the anonymous referees for their careful reading of the manuscript and for their valuable comments and suggestions, which significantly improved the quality of this work. This work is part of Carvalho's Ph.D. thesis at the Department of Mathematics of the Universidade Federal de Pernambuco. It was mostly done during her visit to the Université de Lorraine. She thanks the host institution for its warm hospitality.
		\subsection{Conflict of interest statement}  There are no conflict of interest.

	\end{document}